\documentclass[11pt,a4paper]{amsart}
\usepackage{amsfonts,amscd,amsthm,amsgen,amsmath,amssymb}
\usepackage[all]{xy}
\usepackage[vcentermath]{youngtab}
\newtheorem{theorem}{Theorem}[section]
\newtheorem{proposition}{Proposition}[section]
\newtheorem{lemma}[proposition]{Lemma}

\theoremstyle{definition}
\newtheorem{definition}[theorem]{Definition}
\newtheorem{example}[theorem]{Example}

\theoremstyle{remark}
\newtheorem{remark}[theorem]{Remark}

\numberwithin{equation}{section}

\begin{document}

\title{Conformal Courant algebroids and orientifold T-duality}
\author{David Baraglia}

\address{Mathematical sciences institute, The Australian National University, Canberra ACT 0200, Australia}


\email{david.baraglia@anu.edu.au}

\thanks{This work is supported by the Australian Research Council Discovery Project DP110103745.}


\date{\today}


\begin{abstract}
We introduce conformal Courant algebroids, a mild generalization of Courant algebroids in which only a conformal structure rather than a bilinear form is assumed. We introduce exact conformal Courant algebroids and show they are classified by pairs $(L,H)$ with $L$ a flat line bundle and $H \in H^3(M,L)$ a degree $3$ class with coefficients in $L$. As a special case gerbes for the crossed module $({\rm U}(1) \to \mathbb{Z}_2)$ can be used to twist $TM \oplus T^*M$ into a conformal Courant algebroid. In the exact case there is a twisted cohomology which is $4$-periodic if $L^2 = 1$. The structure of Conformal Courant algebroids on circle bundles leads us to construct a T-duality for orientifolds with free involution. This incarnation of T-duality yields an isomorphism of $4$-periodic twisted cohomology. We conjecture that the isomorphism extends to an isomorphism in twisted $KR$-theory and give some calculations to support this claim.
\end{abstract}

\maketitle


\section{Introduction}

The structure of Courant algebroids makes them useful as a way of capturing the geometry of string theory compactifications. Geometric properties of string theory compactifications can be related to geometric structures on a corresponding Courant algebroid. Courant algebroids also help to clarify the nature of T-duality since much of T-duality (including the Buscher rules \cite{bus}) can be expressed as an isomorphism of Courant algebroids associated to the T-dual spaces. 

Exact Courant algebroids appear in the phase space description of a string propagating in a space $X$ \cite{alst} and this also helps to explain the relation between exact Courant algebroids and gerbes since a gerbe on $X$ with connection and curving determines a ${\rm U}(1)$-valued holonomy over compact oriented surfaces in $X$. Non-exact Courant algebroids are also physically relevant. A process of dimensional reduction produces non-exact Courant algebroids from exact ones. When the dimensional reduction happens over a torus bundle the resulting Courant algebroid admits additional symmetries which can be understood in terms of T-duality. Taking this further Courant algebroids offer a way of understanding so-called non-geometric fluxes and T-folds in a more geometric light.

Despite these selling points Courant algebroids are not the final word in this story. In order to capture the geometry behind compactifications of M-theory or type II string theory with Ramond-Ramond fluxes more complicated bundles that replace the bundle $TX \oplus T^*X$ in generalized geometry have been proposed \cite{hull}. In \cite{bar0} we proposed that these new bundles should have the structure of {\em Leibniz algebroids} and showed that such structure is indeed present. How useful this picture is remains to be seen, but in this paper we will demonstrate the relevance of Leibniz algebroids over Courant algebroids from a completely different source. The structure we introduce we call {\em conformal Courant algebroids} because the bilinear form used in Courant algebroids is replaced by a conformal structure. Despite being only a mild generalization of Courant algebroids the structure of conformal Courant algebroids turns out to be quite interesting. We have found that conformal Courant algebroids are relevant to the T-duality of orientifolds. For simplicity we restrict to the simplest possible class of orientifold, namely those defined by a free involution. This allows us to work with manifolds rather than orbifolds, but we fully expect that our results can be extended to the general orientifold setting.\\

The definition of a conformal Courant algebroid is given in Definition \ref{cca}. Essentially it is a Leibniz algebroid $E$ and a line bundle $L$ equipped with a non-degenerate pairing $\langle \, , \, \rangle : E \otimes E \to L$ and an operator $\nabla : \Gamma(L) \to \Gamma(E)$ (constructed from an $E$-connection on $L$) which satisfy the appropriate generalization of the Courant algebroid axioms. An exact conformal Courant algebroid on $M$ is one for which the natural sequence $0 \to T^*M \otimes L \buildrel \rho^* \over \to  E \buildrel \rho \over \to TM \to 0$ is exact. As explained further in the paper it turns out that conformal Courant algebroids are related to the $E$-Courant algebroids of \cite{cls} and the $AV$-Courant algebroids of \cite{lb}. We classify exact conformal Courant algebroids in Proposition \ref{class}. Isomorphism classes of exact conformal Courant algebroids on $M$ are in bijection with pairs $(L,H)$ where $L$ is a flat line bundle and $H \in H^3(M,L)$ is a degree $3$ cohomology class in the corresponding local system, modulo the equivalence relation $(L,H) \simeq (L,cH)$ for $c \in \mathbb{R}^*$. For applications to orientifolds $L$ should be an orthogonal line bundle corresponding to a class $\epsilon \in H^1(M,\mathbb{Z}_2)$. The class $\epsilon$ determines a double cover $M_\epsilon \to M$ and an involution of $M_\epsilon$, which makes $M_\epsilon$ into a simple kind of orientifold.\\

In Section \ref{twco} we show how there is a notion of twisted cohomology associated to exact conformal Courant algebroids which generalizes cohomology twisted by a closed $3$-form. When $L$ is orthogonal the cohomology is $4$-periodic. In Section \ref{twg} we show that the standard untwisted Courant algebroid structure on $TM \oplus T^*M$ can be twisted by $\epsilon$-twisted graded gerbes. These are a kind of gerbe that generalize bundle gerbes in two ways. First they have a grading which affects the twisted cohomology and second they are gerbes associated to the crossed module $( {\rm U}(1) \buildrel 1 \over \to \mathbb{Z}_2)$ where $\mathbb{Z}_2$ acts on ${\rm U}(1)$ by inversion. This gives a construction for exact conformal Courant algebroids when $L$ is an orthogonal line bundle corresponding to $\epsilon \in H^1(M,\mathbb{Z}_2)$ and $H \in H^3(M,L)$ is integral. The inclusion of a non-trivial class $\epsilon$ corresponds to twisting $TM \oplus T^*M$ by transitions of the form $(X,\xi) \mapsto (X,-\xi)$. This map preserves the standard Courant bracket on $TM \oplus T^*M$ but anti-preserves the natural pairing $\langle \, , \, \rangle$. From the string theory point of view we can think of such transitions as being related to the worldsheet parity reversal operator $\Omega$. We think it is remarkable that this important symmetry in string theory has such a simple interpretation in terms of the Courant bracket on $TM \oplus T^*M$. Gerbes of the form just described can be used to twist $KR$-theory, Atiyah's real $K$-theory of spaces with involution \cite{at}. This makes sense physically since it is thought that on an orientifold Ramond-Ramond flux is classified by a twisted orbifold version of $KR$-theory.\\

In Section \ref{tdot} we investigate T-duality from the point of view of orientifolds and conformal Courant algebroids (again only for the simplest class of orientifold). By demanding that T-duality yields an isormophism of conformal Courant algebroids we are lead to Definition \ref{tdd} which lists the topological conditions that should underlie such an isomorphism. It also generalizes toplogical T-duality for general circle bundles as defined in \cite{bar} which itself is a generalization of topological T-duality in the principal circle bundle case \cite{bem}. It is worth noting that if $\epsilon \neq 0$ then the T-dual torus bundles can not both be principal, therefore incorporating monodromy into T-duality is a necessity for orientifold T-duality. For circle bundles we show in Proposition \ref{tdprop} existence and uniqueness of T-duals. Section \ref{otd} concerns the main properties of this T-duality, namely an isomorphism in twisted cohomology (Proposition \ref{tdciso}) and an isomorphism of conformal Courant algebroids (Proposition \ref{ccaiso}). Finally in Section \ref{twkr} we offer some arguments to support our claim that T-duals also have isomorphic twisted $KR$-theory. It should be possible to prove this using a Fourier-Mukai type of transformation as in \cite{bem}, but we have not attempted to develop the necessary tools (in particular we need a push-forward in twisted $KR$-theory). Despite this we are able to give in Example \ref{ex1} a simple example of orientifold T-duality between the $2$-torus and the Klein bottle which demonstrates the expected isomorphism in $KR$-theory.\\

A few words should be said about what we don't cover. We focus primarily on T-duality in the circle bundle case to avoid extraneous complications that come with higher rank torus bundles. These will be covered in \cite{bar1}, besides which the main focus here is on the move from Courant algebroids to conformal Courant algebroids. We have already stressed that we consider only the simplest possible kind of orientifold, those defined by a free involution. It would be worthwhile to pursue the notion of Courant algebroids and conformal Courant algebroids on orbifolds/orientifolds. With this in place we would expect our formulation of orientifold T-duality carries over. This is definitely not a completely trivial task however. Finally we have not attempted to give a definition of twisted $KR$-theory. When the class $\epsilon = 0$ it reduces to twisted complex $K$-theory. When the $H$-flux $h \in H^3(M,\mathbb{Z}_\epsilon)$ is torsion one can use the definition in \cite{braste} in terms of twisted vector bundles. In the general case only a minor modification from the definition of twisted complex $K$-theory in say \cite{fht} using sections of bundles of Fredholm operators should work, but we have not attempted to carry this out here.


\section{Conformal Courant algebroids}


\subsection{Definitions and basic properties}

\begin{definition}\label{leib}
A {\em Leibniz algebroid} \cite{iban},\cite{bar0} $(E, [ \, , \, ], \rho)$ on $M$ consists of a vector bundle $E$, a vector bundle map $\rho : E \to TM$ called the {\em anchor} and a map $[ \, , \, ] : \Gamma(E) \otimes \Gamma(E) \to \Gamma(E)$ called the {\em Dorfman bracket} such that
\begin{itemize}
\item[(L1)]{$[a,[b,c]] = [[a,b],c] + [b,[a,c]]$,}
\item[(L2)]{$[a,fb] = f[a,b] + \rho(a)(f)b$,}
\item[(L3)]{$\rho[a,b] = [\rho(a),\rho(b)]$,}
\end{itemize}
for all sections $a,b,c$ of $E$ and functions $f$ on $M$.
\end{definition}
Note that axiom (L3) is actually a consequence of the first two axioms \cite{bar0} but is included in the definition for completeness. The Dorfman bracket need not be skew-symmetric so it is generally not a Lie bracket. When it is skew-symmetric we recover the definition of a {\em Lie algebroid} \cite{mac}.\\

Next we introduce the notion of {\em $E$-connections} for a Leibniz algebroid $E$, which generalizes the more familiar case of Lie algebroid connections \cite{fer}.

\begin{definition}
Let $(E, [ \, , \, ], \rho)$ be a Leibniz algebroid. An {\em $E$-connection} on a vector bundle $V$ is an $\mathbb{R}$-linear map $\nabla : \Gamma(V) \to \Gamma(E^* \otimes V)$ such that $\nabla_a (fv) = \rho(a)(f)v + f\nabla_a v$. We say that $V$ is {\em flat} if
\begin{equation}\label{flat}
\nabla_a \nabla_b v -\nabla_b \nabla_a v = \nabla_{[a,b]} v
\end{equation}
for all sections $a,b$ of $E$ and sections $v$ of $V$.
\end{definition}
\begin{remark}
Setting $a=b$ in (\ref{flat}) we see a necessary condition for $\nabla$ to be flat is $\nabla_{[a,a]}v = 0$ for all $a$ and $v$. In general if $\nabla$ is an $E$-connection, not necessarily flat then then expression $\nabla_a \nabla_b v - \nabla_b \nabla_a v - \nabla_{[a,b]} v$ need not be $\mathcal{C}^\infty$-linear in $a$ in contrast with the Lie algebroid case. However in the special case that $\nabla_{[a,a]}v = 0$ for all $a,v$ we find that $\nabla_a \nabla_b v - \nabla_b \nabla_a v - \nabla_{[a,b]} v$ is $\mathcal{C}^\infty$-linear in $a,b,v$ and so defines a section of $\wedge^2 E^* \otimes {\rm End}(V)$ which we might regard as the curvature of $\nabla$.
\end{remark}

We are now ready to introduce the structure which is the main theme of this paper:

\begin{definition}\label{cca}
A {\em conformal Courant algebroid} $(E, [ \, , \, ], \rho, \langle \, , \rangle, L , \nabla)$ on $M$ consists of a Leibniz algebroid $(E, [ \, , \, ], \rho)$, a line bundle $L$ with $E$-connection $\nabla$ and symmetric non-degenerate pairing $\langle \, , \, \rangle : E \otimes E \to L$ such that
\begin{itemize}
\item[(CC1)]{$\nabla_a \langle b , c \rangle = \langle [a,b],c \rangle + \langle b , [a,c] \rangle$,}
\item[(CC2)]{$[a,b] + [b,a] = \nabla \langle a , b \rangle$,}
\end{itemize}
where in (CC2) we use $\langle \, , \, \rangle$ to identify $E^*\otimes L$ with $E$.
\end{definition}
The above definition is original but closely related notions appear in the Courant algebroid literature. It turns out that conformal Courant algebroids are a special case of the $E$-Courant algebroids of \cite{cls}. In the notation of \cite{cls} the bundle $E$ is our line bundle $L$ and the $E$-Courant algebroid anchor $E \to \mathfrak{D}L$ (not to be confused with our anchor $\rho : E \to TM)$ is given by $a \mapsto \nabla_a$. Finally Lemma \ref{flat} below ensures that axiom (EC-1) of \cite{cls} holds. The other axioms are straightforward to verify so that every conformal Courant algebroid is an $E$-Courant algebroid. Also closely related are the $AV$-Courant algebroids of \cite{lb}. In fact the exact conformal Courant algebroids defined in Section \ref{ecca} are precisely the $AV$-Courant algebroids when $A = TM$ is the tangent Lie algebroid and $V = L$ is a line bundle. In \cite{vai} Vaisman introduces a {\em conformal Courant bracket} which is a special case of our definition. Related structures are also found in \cite{wad} and \cite{gg}.\\

If we let $E$ be a Leibniz algebroid and take $L = \mathbb{R}$ with the standard $E$-connection $\nabla_a f = \rho(a)(f)$ then Definition \ref{cca} reduces to the definition of a {\em Courant algebroid} \cite{lwx}, satisfying the axioms for a Courant algebroid presented in \cite[Definition 2.6.1]{roy}.\\

There is a straightforward notion of isomorphism of conformal Courant algebroids, namely $(E,L)$ and $(E',L')$ are isomorphic if there are bundle isomorphisms $E \to E'$, $L \to L'$ exchanging all the structure in the evident manner. 

\begin{lemma}\label{flat}For any conformal Courant algebroid the connection $\nabla$ is flat.
\begin{proof}
Let $a,b,c,d$ be sections of $E$. We have
\begin{eqnarray*}
\nabla_a \nabla_b \langle c , d \rangle &=& \nabla_a \langle [b,c],d \rangle + \nabla_a \langle c , [b,d] \rangle \\
&=& \langle [a,[b,c]] , d \rangle + \langle [b,c],[a,d] \rangle + \langle [a,c],[b,d] \rangle + \langle c , [a,[b,d]] \rangle
\end{eqnarray*}
so that anti-symmetrizing $a,b$ gives
\begin{eqnarray*}
\nabla_a \nabla_b \langle c , d \rangle - \nabla_b \nabla_a \langle c , d \rangle &=& \langle [[a,b],c],d\rangle + \langle c , [[a,b],d] \rangle \\
&=& \nabla_{[a,b]} \langle c,d \rangle
\end{eqnarray*}
and the result follows by non-degeneracy of $\langle \, , \, \rangle$ and the fact that $\nabla_a \nabla_b v - \nabla_b \nabla_a v - \nabla_{[a,b]} v$ is $\mathcal{C}^\infty$-linear in $v$.
\end{proof}
\end{lemma}

\begin{lemma}
For any conformal Courant algebroid we have
\begin{eqnarray}
\left[ \nabla s , a \right] &=& 0, \label{ni1} \\
\left[ a , \nabla s \right] &=& \nabla (\nabla_a s) \label{ni2}.
\end{eqnarray}
\begin{proof}
Let $a,b$ be sections of $E$ and $s$ a section of $L$. By (CC2) we have $[\nabla s , a] = \nabla \langle \nabla s , a \rangle - [a , \nabla s]$ and thus
\begin{eqnarray*}
\langle [\nabla s , a] , b \rangle &=& \langle \nabla \langle \nabla s , a \rangle , b \rangle - \langle [a,\nabla s] , b \rangle \\
&=& \nabla_b \langle \nabla s , a \rangle - \langle [a,\nabla s] , b \rangle \\
&=& \nabla_b \nabla_a s - \nabla_a \langle \nabla s , b \rangle + \langle \nabla s , [a,b] \rangle \\
&=& \nabla_b \nabla_a s - \nabla_a \nabla_b s + \nabla_{[a,b]} s \\
&=& 0
\end{eqnarray*}
since $\nabla$ is flat. This proves (\ref{ni1}) by non-degeneracy of $\langle \, , \, \rangle$. Using this and (CC2) we immediately get (\ref{ni2}).
\end{proof}
\end{lemma}

\begin{lemma}
For any section $s$ of $L$ we have $\rho( \nabla s) = 0$.
\begin{proof}
By non-degeneracy of $\langle \, , \, \rangle $ for any point in $M$ we can locally find sections $a,b$ of $E$ such that $s = \langle a , b \rangle$ then $\rho (\nabla s) = \rho( [a,b] + [b,a]) = [\rho(a),\rho(b)] + [\rho(b),\rho(a)] = 0$.
\end{proof}
\end{lemma}


\subsection{Conformal Courant bracket and Lie $2$-algebra structure}

Let $(E, [ \, , \, ], \rho, \langle \, , \rangle, L , \nabla)$ be a conformal Courant algebroid on $M$. As with ordinary Courant algebroids we can replace the Dorfman bracket $[ \, , \, ]$ by its skew-symmetrization $[ \, , \, ]_C$, namely $[ a , b]_C = \frac{1}{2}( [a,b] - [b,a] )$. We call $[ \, , \, ]_C$ the {\em (conformal) Courant bracket} of $E$. From axiom (CC2) we immediately have $[a,b]_C = [a,b] - \frac{1}{2} \nabla [a,b]$. The price to pay for skew-symmetry is that the Courant bracket does not satisfy the Jacobi identity. However as with Courant algebroids the failure of the Jacobi identity can be neatly collected into the structure of a Lie $2$-algebra. Define an operator $l_3 : \Gamma(E) \otimes \Gamma(E) \otimes \Gamma(E) \to \Gamma(L)$ by
\begin{equation*}
l_3(a,b,c) = \frac{1}{6}( \langle [a,b]_C , c \rangle + \langle [b,c]_C , a \rangle + \langle [c,a]_C , b \rangle ).
\end{equation*}
Then the Jacobi identity is replaced by
\begin{equation*}
[a , [b,c]_C]_C + [b , [c,a]_C ]_C + [c , [a,b]_C ]_C = \nabla l_3(a,b,c).
\end{equation*}
Following \cite{roywei} we put a Lie $2$-algebra structure on the two term complex $\Gamma(L) \buildrel \nabla \over \to \Gamma(E)$. Thought of as an $L_\infty$-algebra with all multilinear brackets beyond the third trivial the structure is as follows. We let $A_1 = \Gamma(E)$ have degree $1$, $A_2 = \Gamma(L)$ degree $2$ and we use the grading convention for $L_\infty$-algebras in which all multilinear brackets $l_k : \Gamma(A_{i_1}) \otimes \dots \otimes \Gamma( A_{i_k}) \to \Gamma(A_{i_1 + \dots + i_k -1})$ have degree $-1$. The non-trivial brackets are determined as follows:
\begin{eqnarray*}
l_1(s) &=& \nabla s, \\
l_2(a,b) &=& [a,b]_C, \\
l_2(a,s) &=& \frac{1}{2} \nabla_a s, \\
l_3(a,b,c) &=& \frac{1}{6}( \langle [a,b]_C , c \rangle + \langle [b,c]_C , a \rangle + \langle [c,a]_C , b \rangle ),
\end{eqnarray*}
where $a,b,c \in \Gamma(E)$, $s \in \Gamma(L)$. The proof that this is an $L_\infty$-algebra is almost identical to the proof in \cite{roywei}.


\subsection{Exact conformal Courant algebroids}\label{ecca}

The notion of exact Courant algebroids is introduced in \cite[letter 1]{sev} along with their classification. Here we extend the definition and classification to conformal Courant algebroids.

\begin{definition}
A conformal Courant algebroid is {\em exact} if the sequence
\begin{equation*}\xymatrix{
0 \ar[r] & T^*M \otimes L \ar[r]^-{\rho^*} & E \ar[r]^-\rho & TM \ar[r] & 0
}
\end{equation*}
is exact. Here $\rho^* : T^*M \otimes L$ denotes the dual $T^*M \to E^*$ of the anchor $\rho : E \to TM$ combined with the identification $E^* \otimes L = E$ given by $\langle \, , \, \rangle$. A {\em splitting} for an exact conformal Courant algebroid is a bundle map $s : TM \to E$ such that $\rho(s(X)) = X$ and $\langle s(X) , s(Y) \rangle = 0$ for all vector fields $X,Y$.
\end{definition}

\begin{lemma}
For an exact conformal Courant algebroid $\nabla_a v = 0$ whenever $\rho(a) = 0$, so there is a uniquely defined operator $\nabla' : \Gamma(L) \to \Gamma(T^*M \otimes L)$ such that $\nabla_a v = \nabla'_{\rho(a)} v$. Moreover $\nabla'$ is a flat connection.
\begin{proof}
Let $m \in M$ and $a \in E_m$ be such that $\rho(a) = 0$. By non degeneracy of $\langle \, , \, \rangle$ and exactness we can find sections $b,c$ of $E$ such that $a = \nabla \langle b , c \rangle (m)$. Now if $v$ is a section of $L$ we get $\nabla_{\nabla \langle b , c \rangle} v = \nabla_{[b,c] + [c,b]} v = 0$ since $\nabla_{[e,e]}v = 0$ for all sections $e$ of $E$. Evaluating at $m$ we see that $(\nabla_a v) (m) = 0$ as required. Thus we may define an operator $\nabla' : \Gamma(L) \to \Gamma(T^*M \otimes L)$ such that $\nabla_{a} = \nabla_{\rho(a)}$. It follows immediately that $\nabla'$ is a flat connection.
\end{proof}
\end{lemma}

By abuse of notation we will use $\nabla$ to denote the flat connection defined in the above Lemma for exact conformal Courant algebroids.

\begin{lemma}
Every exact conformal Courant algebroid admits a splitting.
\begin{proof}
First we show that the image of $T^*M \otimes L$ in $E$ is isotropic. For any $m \in M$ and $\xi,\eta \in (T^*M \otimes L)_m$ choose sections $f,g$ of $L$ such that $\nabla f(m) = \xi$, $\nabla g(m) = \eta$. Then $\langle \nabla f , \nabla g \rangle = \nabla_{\nabla f} (g) = 0$ since $\rho( \nabla f) = 0$.

We have that $T^*M \otimes L$ is an isotropic subspace of $E$ and by exactness the dimension of $T^*M \otimes L$ is exactly half the dimension of $E$. Thus $T^*M \otimes L$ is a maximal isotropic.\\

Now choose any map $t : TM \to E$ such that $\rho \circ t = {\rm id}$. Such a map exists because exact sequences of smooth vector bundle can always be split. In general $\langle t(X) , t(Y) \rangle$ will be a non-trivial $L$-valued symmetric form on $TM$. Observe that if $\xi \in T^*M \otimes L$ then $\langle t(X) , \xi \rangle = \xi(X)$ (because we use the dual of the anchor $\rho$ to identify $T^*M \otimes L$ as a subbundle of $E$). For any bundle map $\phi : TM \to T^*M \otimes L$ we may define a new section $s : TM \to E$ given by $s(X) = t(X) + \phi(X)$. One finds that $\langle s(X) , s(Y) \rangle = \langle t(X) , t(Y) \rangle + \phi(X)(Y) + \phi(Y)(X)$. It follows that on suitable choice of $\phi$ the image $s(TM)$ of $TM$ under $s$ is isotropic so that $s$ is a splitting.
\end{proof}
\end{lemma}

Let us now fix a choice of splitting $s : TM \to E$ for an exact conformal Courant algebroid. This fixes an identification of $E$ with $TM \oplus (T^*M \otimes L)$ with pairing
\begin{equation}\label{pair}
\langle X+\xi , Y + \eta \rangle = i_X \eta + i_Y \xi
\end{equation}
where $X,Y$ are valued in $TM$ and $\xi,\eta$ valued in $T^*M \otimes L$.

\begin{proposition}
Let $(E,L)$ be an exact conformal Courant algebroid. After choosing a splitting $s : TM \to E$ and using the identification $E \simeq TM \oplus (T^*M \otimes L)$ corresponding to $s$ there exists an $L$-valued $3$-form $H \in \Gamma( \wedge^3 T^*M \otimes L)$ such that $d_{\nabla} H = 0$ and
\begin{equation}\label{cdb}
[ X + \xi , Y + \eta] = [X,Y]_{{\rm Lie}} + \mathcal{L}^{\nabla}_X \eta - i_Y d_{\nabla} \xi + i_Y i_X H
\end{equation}
where $[X,Y]_{{\rm Lie}}$ is the Lie bracket of vector fields and $\mathcal{L}^{\nabla}_X = i_X d_{\nabla} + d_{\nabla} i_X$. Conversely for any choice of line bundle $L$ with flat connection $\nabla$ and $d_{\nabla}$-closed $L$-valued $3$-form $H$ we get an exact conformal Courant algebroid $TM \oplus (T^*M \otimes L)$ with pairing $\langle \, , \, \rangle$ given by (\ref{pair}) and Dorfman bracket given by (\ref{cdb}).
\begin{proof}
First let us remark that since $\nabla$ is flat the familiar Cartan relations hold, that is $[\mathcal{L}_X^{\nabla} , i_Y] = i_{[X,Y]_{{\rm Lie}}}$, $[\mathcal{L}_X^{\nabla} , d_{\nabla} ] = 0$ and by definition $\mathcal{L}_X^{\nabla} = i_X d_{\nabla} + d_{\nabla} i_X$. To see this simply observe that the result needs only be checked locally where we can use a constant section of $L$ to identify $\nabla$ with the trivial connection.\\

Let $X,Y,Z$ be sections of $TM$ and $\xi,\eta$ sections of $T^*M \otimes L$. Since $\rho [ X , \eta ] = 0$ we have that $[X , \eta]$ is valued in $T^*M \otimes L$ so to determine an expression for $[X,\eta]$ it suffices to evaluate $\langle [X, \eta] , Z \rangle$. One finds
\begin{eqnarray*}
\langle [X , \eta] , Z \rangle &=& \nabla_X \langle \eta , Z \rangle - \langle \eta , [X,Z] \rangle \\
&=& \nabla_X (i_Z \eta) - i_{[X,Z]_{{\rm Lie}}} \eta \\
&=& i_X d_{\nabla} i_Z \eta - i_{[X,Z]_{{\rm Lie}}} \eta \\
&=& \mathcal{L}_X^{\nabla} i_Z \eta - i_{[X,Z]_{{\rm Lie}}} \eta \\
&=& i_Z \mathcal{L}_X^{\nabla} \eta.
\end{eqnarray*}
Thus $[X,\eta] = \mathcal{L}_X^{\nabla} \eta$. Now using $[X,\eta] + [\eta, X] = \nabla \langle X , \eta \rangle$ we immediately obtain $[\xi , Y] = -i_Y d_{\nabla} \xi$.

We have $\rho[ \xi , \eta] = 0$ so as above to evaluate $[\xi,\eta]$ it suffices to consider $\langle [\xi,\eta] , Z \rangle$. We have
\begin{eqnarray*}
\langle [\xi,\eta] , Z \rangle &=& \nabla_\xi \langle \eta , Z \rangle - \langle \eta , [\xi,Z] \rangle \\
&=& \langle \eta , i_Z d_{\nabla} \xi \rangle \\
&=& 0
\end{eqnarray*}
so $[\xi , \eta ] =0$. Finally we must consider $[X,Y]$. Since $\rho[X,Y] = [\rho(X),\rho(Y)]_{{\rm Lie}}$ we find that $[X,Y] = [X,Y]_{{\rm Lie}} + H(X,Y)$ for some bundle map $H : TM \otimes TM \to T^*M \otimes L$. Clearly $H(X,Y) = -H(Y,X)$. Using $\langle [X,Y] , Z \rangle + \langle Y , [X,Z] \rangle = 0$ we further see that $\langle H(X,Y) , Z \rangle + \langle Y , H(X,Z) \rangle = 0$. Thus if we think of $H$ as valued in $T^*M \otimes T^*M \otimes T^*M \otimes L$ then $H$ is skew-symmetric, thus a section of $\wedge^3 T^*M \otimes L$. After expanding the Leibniz identity $[X,[Y,Z]] = [[X,Y],Z] + [Y,[X,Z]]$ and simplifying one finds $d_{\nabla} H = 0$.\\

Conversely if $H$ is a $d_{\nabla}$-closed $L$-valued $3$-form and we define a bracket on $TM \oplus (T^*M \otimes L)$ by (\ref{cdb}) then it is a matter of straightforward computation to show it satisfies the axioms for a conformal Courant algebroid.
\end{proof}
\end{proposition}

\begin{definition} Let $L$ be a line bundle with flat connection. The bundle $E = TM \oplus (T^*M \otimes L)$ with pairing $E \otimes E \to L$ given by (\ref{pair}) will the called the {\em $L$-twisted generalized tangent bundle} of $M$. If we wish to emphasize the dependence of $E$ on $L$ we will write $E_L$. If $L$ corresponds to a class $\epsilon \in H^1(M,\mathbb{R}^*)$ we may also write $E_\epsilon$ and call it the {\em $\epsilon$-twisted generalized tangent bundle}.
\end{definition}

\begin{definition} Let $L$ be a line bundle with flat connection and let $H$ be a closed $L$-valued $3$-form on $M$. Then the $L$-twisted generalized tangent bundle $E_L$ admits the structure of a conformal Courant algebroid with Dorfman bracket determined by $H$ as in Equation (\ref{cdb}). We call this Dorfman bracket the {\em $H$-twisted Dorfman bracket} on $E_L$. We denote the bracket by $[ \, , \, ]$ or by $[ \, , \, ]_H$ to emphasize the dependence on $H$.
\end{definition}

In the case $L = \mathbb{R}$ with trivial connection the $L$-twisted generalized tangent bundle $E = TM \oplus T^*M$ will also called the {\em generalized tangent bundle} and if $H$ is a closed $3$-form then the $H$-twisted Dorfman bracket gives $E$ the structure of a Courant algebroid. The $H$-twisted Courant bracket in this case can be found in \cite{sw}.

\begin{proposition}
If $s : TM \to E$ is a splitting for an exact conformal Courant algebroid $(E,L)$ then any other splitting $t : TM \to E$ has the form $t(X) = s(X) + i_X B$ where $B \in \Gamma( \wedge^2 T^*M \otimes L)$ is an $L$-valued $2$-form. Under such a change the $3$-form $H$ changes to $H + d_{\nabla}B$.
\begin{proof}
That any other splitting $t$ has the given form is straightforward. Now if $H$ is defined by $[s(X),s(Y)] = s([X,Y]_{{\rm Lie}}) + i_X i_Y H$ and $H'$ defined by $[t(X),t(Y)] = t([X,Y]_{{\rm Lie}}) + i_X i_Y H'$ then we find easily that $H' = H + d_{\nabla} B$.
\end{proof}
\end{proposition}

\begin{proposition}\label{class}
Isomorphism classes of exact conformal Courant algebroids on $M$ are in bijection with pairs $(\epsilon,h)$ where $\epsilon \in H^1(M,\mathbb{R}^*)$ determines an isomorphism class of line bundle $L$ with flat connection and $h \in H^3(M,L)$ is degree $3$-class in de Rham cohomology twisted by $L$ modulo the equivalence relation $(\epsilon,h) \sim (\epsilon,h')$ where $h' = ch$ for some constant $c \in \mathbb{R}^*$.
\begin{proof}
The only thing that needs clarification is the equivalence relation $(\epsilon,h) \sim (\epsilon,h')$ where $h' = ch$ for some constant $c \in \mathbb{R}^*$. We can think of $c$ as a bundle endomorphism $c : L \to L$ that preserves the flat connection on $L$. Conversely any isomorphism of $L$ that preserves a connection on $L$ must be a constant.
\end{proof}
\end{proposition}

If we take $L$ to be the trivial line bundle $M \times \mathbb{R}$ with trivial connection and take $H = 0$ then $E = TM \oplus T^*M$ with bracket (\ref{cdb}) will be called the {\em standard conformal Courant algebroid} on $TM \oplus T^*M$. In this case since $L$ is the trivial line bundle $E$ is actually a Courant algebroid and will also be called the {\em standard Courant algebroid}. It follows from Proposition \ref{class} that in a contractible neighborhood $U \subset M$ there is up to isomorphism only one exact conformal Courant algebroid, namely $TM \oplus T^*M$ with the standard Dorfman bracket.\\

Let $(E,L)$ be any exact conformal Courant algebroid on $M$. Then we can find an open cover $\{ U_\alpha \}$ of $M$ and isomorphisms $\phi_\alpha : (TM\oplus T^*M,1)|_{U_\alpha} \to (E,L)|_{U_\alpha}$ where $(TM \oplus T^*M , 1)$ denotes the standard Courant algebroid. Note that it makes sense to restrict a conformal Courant algebroid to an open subset as the Dorfman bracket is seen to be a differential operator\footnote{Note however that this is not true for general Leibniz algebroids, that is there are Leibniz algebroids with Dorfman bracket that is not a differential operator.}. On double intersections $U_{\alpha \beta} = U_\alpha \cap U_\beta$ we have that $g_{\alpha \beta} = \phi_\alpha^{-1} \circ \phi_\beta$ is an automorphism of the standard conformal Courant algebroid $(TM \oplus T^*M,1)|_{U_{\alpha \beta}}$ on $U_{\alpha \beta}$ and that the $\{ g_{\alpha \beta} \}$ forms a $1$-cocycle with values in the (sheaf of) automorphisms of the standard conformal Courant algebroid. Note that by automorphism we mean an automorphism covering the identity on $M$. It is well known that such automorphisms of $TM \oplus T^*M$ as a Courant algebroid are given by closed $2$-forms \cite{gual}, namely if $B$ is a closed $2$-form then $B$ defines a transformation $e^B(X+\xi) = X + \xi + i_X B$ of $TM \oplus T^*M$ preserving the Courant bracket. However as a {\em conformal Courant algebroid} $TM \oplus T^*M$ admits a slightly bigger group of automorphisms, namely for any $c \in \mathbb{R}^*$ we have a conformal shift $(X,\xi) \mapsto (X,c\, \xi)$ which is easily seen to preserve the Dorfman bracket but only preserves the conformal class of $\langle \, , \, \rangle$. Together $B$-shifts and conformal shifts combine to give the full group of automorphisms which is a semi-direct product $\mathbb{R}^* \ltimes \Omega^2_{{\rm cl}}(M)$ where $\Omega^2_{{\rm cl}}(M)$ denotes the space of closed $2$-forms which is a group under addition and $\mathbb{R}^*$ acts by multiplication. It is not hard to see that the \v{C}ech cohomology class of a $1$-cocycle $\{ g_{\alpha \beta} \}$ with values in the sheaf $\mathbb{R}^* \ltimes \Omega^2_{{\rm cl}}(M)$ corresponds to a pair $(\alpha,h)$ modulo equivalence $(\epsilon,h) \sim (\epsilon,h')$ exactly as specified in Proposition \ref{class}. Of course this is exactly as it should be since equivalence classes of such $1$-cocycles $\{ g_{\alpha \beta} \}$ correspond isomorphism classes of exact conformal Courant algebroids.


\subsection{Twisted cohomology}\label{twco}

Let $(E,L)$ be an exact conformal Courant algebroid on $M$. Choose a splitting $s : TM \to E$ so that $E$ is identified with $TM \oplus (T^*M \otimes L)$ with pairing (\ref{pair}) and Dorfman bracket \ref{cdb} defined by a $d_\nabla$-closed $L$-valued $3$-form $H$. We define $S_L^i(M) = \bigoplus_{j \in \mathbb{Z}} L^j \otimes \wedge^{i+2j}T^*M$. There is a natural action $\gamma : E \otimes S^i_L(M) \to S^{i-1}_L(M)$ of $E$ given by
\begin{equation*}
\gamma_{X+\xi}  \omega = i_X \omega + \xi \wedge \omega.
\end{equation*}
The action similar to that of a Clifford module but differs in that $\langle \, , \, \rangle$ is $L$-valued rather than $\mathbb{R}$-valued. There is also a natural action $L \otimes S^i_L(M) \to S^{i-2}_L(M)$ which commutes with $\gamma$ and satisfies $\gamma_a \gamma_b + \gamma_b \gamma_a = \langle a , b \rangle$.\\

We define a differential $d_{\nabla,H} : \Gamma(S^i_L(M)) \to \Gamma(S^{i+1}_L(M))$ by $d_{\nabla,H} \omega = d_\nabla \omega + H \wedge \omega$. We observe immediately that $d_{\nabla , H}^2 = 0$ and so defines cohomology groups which we denote by $H^{i,H}_L(M)$. If one changes the splitting $s$ by an $L$-valued $2$-form $B$ then $H$ changes to $H + d_\nabla B$, but one easily observes that $e^{-B} \circ \, d_{\nabla,H} \circ \, e^{B} =  d_{\nabla,H+d_\nabla B}$ where $e^B$ acts on $S^i_L(M)$ by $e^B \omega = \omega + B \wedge \omega + \frac{1}{2} B \wedge B \wedge \omega + \dots $, thus there is an induced isomorphism $e^{-B} : H^{i,H}_L(M) \to H^{i,H+d_\nabla B}_L(M)$. This shows that the cohomology groups $H^{i,H}_L(M)$ are up to isomorphism independent of the choice of splitting and are thus associated only to the isomorphism class of the conformal Courant algebroid $(E,L)$. In fact for any two splittings $s,s' : TM \to E$ there is a {\em unique} $2$-form $B$ relating the splittings and thus a canonical isomorphism $e^{-B}$ between the twisted cohomology groups. We will call $H^{i,H}_L(M)$ the {\em twisted cohomology} associated to the conformal courant algebroid $(E,L)$.\\

For a given section $a$ of $E$ let us define an operator $\mathcal{L}_a : \Gamma(S^i_L(M)) \to \Gamma(S^i_L(M))$ by the usual expression $\mathcal{L}_a(\omega) = \gamma_a (d_{\nabla,H}\omega) + d_{\nabla,H}( \gamma_a \omega)$.
\begin{proposition}
We have the following identities
\begin{eqnarray}
\left[\mathcal{L}_a , d_{\nabla,H}\right] &=& 0, \label{id1} \\
\left[\mathcal{L}_a , \gamma_b\right] &=& \gamma_{[a,b]_H}, \label{id2} \\
\left[\mathcal{L}_a , \mathcal{L}_b \right] &=& \mathcal{L}_{[a,b]_H}. \label{id3}
\end{eqnarray}
\begin{proof}
Equation (\ref{id1}) follows from the definition of $\mathcal{L}_a$ and the fact that $d_{\nabla,H}^2 = 0$. For (\ref{id2}) we first find that if $a = X + \xi$ and $\omega \in \Gamma(S^i_L(M))$ then $\mathcal{L}_a \omega = \mathcal{L}_X \omega + d\xi \wedge \omega + (i_X H) \wedge \omega$. From this the calculation is straightforward. Equation (\ref{id3}) follows immediately using (\ref{id1}) and (\ref{id2}) applied to $[\mathcal{L}_a , \mathcal{L}_b] = [\mathcal{L}_a , [d_{\nabla,H},\gamma_b]]$, where $[d,\gamma_b]$ denotes the anti-commutator $\gamma_b d_{\nabla,H} + d_{\nabla,H} \gamma_b$.
\end{proof}
\end{proposition}
As $a$ runs over the sections of $E$ the operators $\mathcal{L}_a , \gamma_a$ together with $d_{\nabla,H}$ generate a graded Lie algebra and the Dorfman bracket can be thought of as a derived bracket in the sense that 
\begin{equation}\label{db}
\gamma_{[a,b]_H} = [ [d_{\nabla,H} , \gamma_a] , \gamma_b ]
\end{equation}
for any sections $a,b$ of $TM \oplus L \otimes T^*M$.\\

If $L$ is the trivial line bundle with trivial connection then we see immediately that $H^{i+2k,H}_L(M) = H^{i,H}_L(M)$ so the twisted cohomology is $2$-periodic and in fact coincides with the usual definition of twisted cohomology twisted by a closed $3$-form. If $L$ is not trivial but is orthogonal, so it is associated to a class in $H^1(M,\mathbb{Z}_2)$ then $L^2$ is trivial as a line bundle with connection and it follows that $H^{i+4k,H}_L(M) = H^{i,H}_L(M)$, so twisted cohomology is $4$-periodic in this case. It is for this case that the theory of conformal Courant algebroids and twisted cohomology will make contact with twisted $KR$-theory.\\

The differential $d_{\nabla,H}$ preserves the filtration 
\begin{equation*}
F^p( S^n_L(M)) = \bigoplus_{j \ge (p-n)/2} L^j \otimes \wedge^{n+2j} T^*M
\end{equation*}
and so there is an associated spectral sequence $(E_r^{p,q},d_r)$ which abuts to the twisted cohomology $H^{n,H}_L(M)$. One easily sees that $E_3^{p,q} = 0$ if $q$ is odd and $E_3^{p,q} = H^p(M,L^{q/2})$ if $q$ is even, that is degree $p$ de Rham cohomology twisted by $L^{q/2}$. The differential $d_3$ is given by the wedge product $[H]\wedge : H^p(M,L^{q/2}) \to H^{p+3}(M,L^{q/2 + 1})$.\\

For application to $T$-duality we will want to consider a further twist of de Rham cohomology. Let $\alpha \in H^1(M,\mathbb{R}^*)$ and represent $\alpha$ as a line bundle with flat connection $(A , \nabla^\alpha)$. We define $S^{i,A}_L(M) = \bigoplus_{j \in \mathbb{Z}} A \otimes L^j \otimes \wedge^{i+2j} T^*M$ and differential $d_{\nabla,H,A}$ given by combining the previously defined differential $d_{\nabla,H}$ with the flat connection $\nabla^\alpha$ in the usual fashion. The relations (\ref{id1})-(\ref{id3}) still apply in this situation. We let $H^{i,A,H}_L(M)$ or $H^{i,\alpha,H}_L(M)$ denote the corresponding cohomology groups. We observe that there are canonical isomorphisms $H^{i,A,H}_L(M) = H^{i+2,A \otimes L,H}_L(M)$. In particular twisting by $A = L^*$ corresponds to a degree shift by two. Finally note that the filtration and spectral sequence extend to the more general case with $E_3^{p,q} = 0$ if $q$ is odd and $E_3^{p,q} = H^p(M,A \otimes L^{q/2})$ if $q$ is even.


\subsection{Twisting by a gerbe}\label{twg}

In a well known construction \cite{gual} an abelian gerbe with connection and curving can be used to twist the standard Courant algebroid $TM \oplus T^*M$ into an exact Courant algebroid such that the isomorphism class of the Courant algebroid in $H^3(M,\mathbb{R})$ is integral, indeed it is the image of the Diximier-Douady class of the gerbe in real cohomology. A similar construction using a slightly more general type of gerbe will allow us to twist $TM \oplus T^*M$ to a conformal Courant algebroid with class $(\epsilon,h)$ where $\epsilon \in H^1(M,\mathbb{Z}_2) \subseteq H^1(M,\mathbb{R}^*)$ and $h$ lies in the image $H^3(M,\mathbb{Z}_\epsilon) \to H^3(M,\mathbb{R}_\epsilon)$. The type of gerbe we use correspond to the crossed module $({\rm U}(1) \to \mathbb{Z}_2)$, where $\mathbb{Z}_2$ acts on ${\rm U}(1)$ by inversion while the homomorphism ${\rm U}(1) \to \mathbb{Z}_2$ is trivial. For a given class $\epsilon \in H^1(M,\mathbb{Z}_2)$ such gerbes will be called {\em $\epsilon$-twisted (bundle) gerbes} or {\em $\epsilon$-twisted graded (bundle) gerbes} if it is equipped with a grading as discussed below. These gerbes can also be understood as a special case of Jandl gerbes \cite{fnsw} where the involution is free, except we also allow for a grading. We understand such gerbes and their extension to the orbifold setting to be the objects which can be used to twist $KR$-theory.\\

To expedite the construction we will consider only $\epsilon$-twisted gerbes defined with respect to a good open cover $\mathcal{U} = \{ U_i \}$ of $M$. Since any gerbe is stably isomorphic to one defined with respect to $\mathcal{U}$ this is not a serious restriction. By a good cover we mean that the multiple intersections $U_{i_1  i_2  \dots  i_k} = U_{i_1} \cap U_{i_2} \cap \dots \cap U_{i_k}$ are contractible. This is not necessary but prevents us from having to take subcovers later.

To define $\epsilon$-twisted gerbes with respect to $\mathcal{U}$ we must first choose a $\mathbb{Z}_2$-valued cocycle $\{ \epsilon_{ij} \}$ representing $\epsilon$. The definition of $\epsilon$-twisted bundle gerbes depends on the choice of representative $\{ \epsilon_{ij} \}$, but it is straightforward to related the definitions for different choices of representative. To specify an $\epsilon$-twisted gerbe with respect to $\mathcal{U}$ we must give a collection of Hermitian line bundles $\{L_{ij}\}$ with $L_{ij}$ defined over $U_{ij}$. We say the gerbe is {\em graded} if in addition the $L_{ij}$ are given a $\mathbb{Z}_2$-grading by specifying a $\mathbb{Z}_2$-valued cocycle $\{ \alpha_{ij} \}$ It turns out that this additional structure is irrelevant to the construction of the conformal Courant algebroid, so we ignore it for the time being. The $\{ \alpha_{ij} \}$ will re-appear when we consider twisted cohomology. The structure of the gerbe is given by a product
\begin{equation*}
\theta_{ijk} : L_{ij}^{\epsilon_{jk}} \otimes L_{jk} \to L_{ik}
\end{equation*}
where $L_{ij}^{1} = L_{ij}$ and $L_{ij}^{-1} = L_{ij}^*$. The product $\theta$ is an isomorphism of Hermitian line bundles for any $i,j,k$ with $U_{ijk}$ non-empty. The multiplication is required to satisfy an associativity condition. Rather than write this out in full we will pass to a \v{C}ech description. Since $\mathcal{U}$ is taken to be a good cover we can find sections $\sigma_{ij} : U_{ij} \to L_{ij}$ of unit norm. Then we get ${\rm U}(1)$-valued functions $g_{ijk} : U_{ijk} \to {\rm U}(1)$ defined by $\theta_{ijk}(\sigma_{ij}^{\epsilon_{jk}} , \sigma_{jk}) = g_{ijk} \sigma_{ik}$ where $\sigma_{ij}^{\epsilon_{jk}}$ is the section of $L_{ij}^{\epsilon_{jk}}$ corresponding to $\sigma_{ij}$. The associativity condition translates into a twisted cocycle condition for the $g_{ijk}$:
\begin{equation*}
g_{ijl} g_{jkl} = \epsilon_{kl}( g_{ijk} ) g_{ikl}
\end{equation*}
where $\mathbb{Z}_2$ acts on ${\rm U}(1)$ by letting the non-trivial element of $\mathbb{Z}_2$ act as inversion. The class $\{ g_{ijk} \}$ is a \v{C}ech $2$-cocycle defining a class in $ \linebreak H^2(M , \mathcal{C}^{\infty}({\rm U}(1)_{\epsilon} ) )$ where $\mathcal{C}^{\infty}({\rm U}(1)_{\epsilon} )$ is the sheaf obtained from $\mathcal{C}^{\infty}({\rm U}(1) )$ by twisting by the cocycle $\epsilon = [\{ \epsilon_{ij} \}]$. As in the untwisted case the coboundary $\delta : H^2(M , \mathcal{C}^{\infty}({\rm U}(1)_{\epsilon} ) ) \to H^3(M , \mathbb{Z}_\epsilon )$ is an isomorphism, where $\mathbb{Z}_\epsilon$ is the local system with $\mathbb{Z}$-coefficients determined by $\epsilon$. Thus associated to such a gerbe $\mathcal{G} = (\{\theta_{ijk}\} , \{L_{ij}\})$ is a class $DD(\mathcal{G}) \in H^3(M,\mathbb{Z}_\epsilon)$ which is the equivalent of the Diximier-Douady class and determines the stable isomorphism class of $\mathcal{G}$. If we allow a grading $\alpha_{ij}$ then the triple $\mathcal{G} = (\{\theta_{ijk}\},\{L_{ij}\},\{\alpha_{ij}\})$ is classified by the pair $([\alpha] , DD(\mathcal{G}) ) \in H^1(M,\mathbb{Z}_2) \oplus H^3(M,\mathbb{Z}_\epsilon) $.\\

Next we equip $\mathcal{G}$ with a {\em connection}, that is each $L_{ij}$ is equipped with a connection $\nabla_{ij}$ such that the multiplication $\theta_{ijk}$ is constant. Define $1$-forms $A_{ij}$ by $\nabla_{ij}(\sigma_{ij}) = A_{ij} \sigma_{ij}$. Then from the definition of $g_{ijk}$ we obtain by differentiating the following: 
\begin{equation*}
\epsilon_{jk} A_{ij} + A_{jk} = A_{ik} + d {\rm log}( g_{ijk} ).
\end{equation*}
Let $F_{ij} = dA_{ij}$ be the curvature of $\nabla_{ij}$ then
\begin{equation*}
\epsilon_{jk} F_{ij} + F_{jk} = F_{ik}.
\end{equation*}
By the usual sort of argument for fine sheaves we know that there exists $2$-forms $B_i$ such that
\begin{equation}\label{feq}
F_{ij} = B_j - \epsilon_{ij} B_i.
\end{equation}
A choice of such $\{ B_i \}$ is called a {\em curving} for the connection. Now set $H_i = dB_i$. We find
\begin{equation*}
H_j = \epsilon_{ij} H_i
\end{equation*}
so $\{ H_i \}$ defines a class in $H^3(M,\mathbb{R}_\epsilon)$, called the {\em curvature} of the curving. Up to some universal constant depending on our conventions we have that the class of $H$ represents this image of $DD(\mathcal{G})$ in $H^3(M,\mathbb{R}_\epsilon)$.\\

Now we explain how the gerbe $\mathcal{G}$ along with connection $\{ \nabla_{ij} \}$ and curving $\{ B_i \}$ can be used to twist the standard Courant algebroid $(TM \oplus T^*M , 1)$ into an exact conformal Courant algebroid. We patch together $(TM \oplus T^*M)|_{U_j}$ to $(TM \oplus T^*M)|_{U_i}$ by the following transition relation on the overlap $U_{ij}$
\begin{eqnarray*}
X_i &=& X_j, \\
\xi_i &=& \epsilon_{ij}( \xi_j + i_{X_j} F_{ij} ).
\end{eqnarray*}
These transitions are by automorphisms of the standard Courant algebroid $TM \oplus T^*M$ so we obtain an exact conformal Courant algebroid $(E,\mathbb{R}_\epsilon)$, where $\mathbb{R}_\epsilon$ is the flat line bundle associated to $\epsilon$. But now by (\ref{feq}) we immediately obtain the relation
\begin{equation*}
\xi_i + i_{X_i}B_i = \epsilon_{ij}( \xi_j + i_{X_j} B_j ).
\end{equation*}
Thus the choice of $\{ B_i \}$ determines an isomorphism $\phi : E \to TM \oplus (T^*M \otimes \mathbb{R}_\epsilon)$. To work out what the bracket on $E$ is under the identification $\phi$ we need only note that $X_i + (\xi_i + i_{X_i}B_i) = e^{B_i}(X_i + \xi_i)$, so if $[ \, , \, ]$ is the standard bracket on $TM \oplus T^*M$ then the twisted bracket over $U_i$ is given by $e^{-B_i}[ e^{B_i}(X+\xi) , e^{B_i}(Y+\eta)] = [X+\xi , Y+\eta] + i_Y i_X dB_i$, so the bracket we get identifies with the the $H_i$-twisted bracket on $TM \oplus (T^*M \otimes \mathbb{R}_\epsilon)$.\\

Suppose now our gerbe $\mathcal{G} = (\{\theta_{ijk}\},\{L_{ij}\})$ with connection and curving is further equipped with a grading $\{ \alpha_{ij} \}$. We can use the data of $\mathcal{G}$ with connection, curving and grading to twist the bundles $\wedge^{ev / odd} T^*M$ of even and odd forms. More specifically for $r \in \mathbb{Z}$ we define a twisted bundle $S^r(\mathcal{G})$ of forms as follows. Over each open set $U_i$ we identify $S^r(\mathcal{G})$ with either the bundle $\wedge^{ev} T^*M |_{U_i}$ of even forms if $r$ is even or the bundle $\wedge^{odd} T^*M |_{U_i}$ of odd forms if $r$ is odd. A section of $S^r(\mathcal{G})$ is given by a collection $\{ \omega_i \}$ where $\omega_i \in \Omega^{ev/odd}(U_i)$ is an even or odd form on $U_i$ according to the parity of $r$ and on the overlap $U_{ij}$ we have the relation $\omega_i =\epsilon_{ij} \circ ( \alpha_{ij} e^{-F_{ij}} \omega_j)$, where if $\mu$ has degree $d \in r + 2\mathbb{Z}$ then $1 \circ \mu = \mu$ and $(-1) \circ \mu = (-1)^{(r-d)/2} \mu$. 

Using (\ref{feq}) we immediately obtain the identity $e^{-B_i}\omega_i = \epsilon_{ij} \circ ( \alpha_{ij} e^{-B_j} \omega_j)$. Thus the collection of forms $\{ e^{-B_i} \omega_i \}$ patch together to form a section of $S^{r,A}_L(M) =  \bigoplus_{k \in \mathbb{Z}} A \otimes L^k \otimes \wedge^{r+2k} T^*M$ where $A$ is the flat line bundle corresponding to $\{ \alpha_{ij} \}$ and $L$ the flat line bundle corresponding to $\{ \epsilon_{ij} \}$. Thus a choice of connection and grading for $\mathcal{G}$ let us define the bundles $S^r(\mathcal{G})$ and a choice of curving for the connection yields an isomorphism $S^r(\mathcal{G}) \to S^{r,A}_L(M)$. The bundles $S^r(\mathcal{G})$ come with some additional structure which under this isomorphism $S^r(\mathcal{G}) \to S^{r,A}_L(M)$ coincides with the structures $\gamma, d_{\nabla,H,A}$ defined in Section \ref{twco}. First we define a Clifford action $E \otimes S^r(\mathcal{G}) \to S^{r-1}(\mathcal{G})$ as follows. Over each open subset $U_i$ of the cover we can identify $S^r(\mathcal{G})$ with the bundle of even or odd forms and $E$ with $TM \oplus T^*M$ and we take $\gamma$ to be defined by $\gamma_{X+\xi} \omega = i_X \omega + \xi \wedge \omega$ in $U_i$. One checks easily that this action coincides on the overlaps $U_{ij}$ and under the identifications $E = TM \oplus (T^*M \otimes L)$, $S^r(\mathcal{G}) = S^{r,A}_L(M)$ we see that $\gamma$ coincides with the definition in Section \ref{twco}. Similarly there is a differential operator $D : \Gamma( S^r(\mathcal{G})) \to \Gamma( S^{r+1}(\mathcal{G}))$. Over an open set $U_i$ we identify $S^r(\mathcal{G})$ with the bundle of even or odd forms as usual and define $D$ to be the exterior derivative. One sees easily that the definition of $D$ agrees on overlaps $U_{ij}$ so that $D$ is well-defined. Under the identifications $S^r(\mathcal{G}) = S^{r,A}_L(M)$ we see that $D$ is sent to the operator $d_{\nabla,H,A}$ of Section \ref{twco}.


\subsection{Dirac structures}

For a Courant algebroid a Dirac structure is a maximal isotropic subbundle such that the space of sections is closed under the Dorfman bracket (or equivalently the Courant bracket). Since the notion of a maximal isotropic subspace depends only on the conformal class of the pairing this notion immediately extends to conformal Courant algebroids.
\begin{definition} An {\em almost Dirac structure} for a conformal Courant algebroid $E$ is a maximal isotropic subbundle $D \subset E$. An almost Dirac structure $D$ is called {\em integrable} and said to be a {\em Dirac structure} if the space of sections of $D$ is closed under the Dorfman bracket.
\end{definition}

This definition appears to have some overlap with the conformal Dirac structures of \cite{wad}, although the exact relation is unclear. From the above definition it immediately follows that a Dirac structure equipped with the restriction of the Dorfman bracket and anchor has the structure of a Lie algebroid.\\

In the case of the standard Courant algebroid $E = TM \oplus T^*M$ on a manifold $M$ Dirac structures provide a simultaneous generalization of pre-symplectic structures (closed $2$-forms) and Poisson structures. Taking a closed $3$-form $H$ and using the $H$-twisted Dorfman bracket on $E$ we get the corresponding $H$-twisted notions. Now suppose $L$ is a line bundle with flat connection $\nabla$ and consider $E = TM \oplus (T^*M \otimes L)$ with the standard conformal Courant bracket of Equation (\ref{cdb}) with $H = 0$. We consider two kinds of Dirac structures generalizing the pre-symplectic and Poisson cases when $L$ is trivial.

\begin{example} Suppose $D$ is a graph of an $L$-valued $2$-form $\omega \in \Omega^2(M,L)$, that is $D = \{ X + i_X \omega \, | \, X \in TM \}$. Integrability of $D$ is easily seen to be the requirement that $\omega$ is $d_{\nabla}$-closed. As a special case suppose that $L$ is orthogonal so corresponds to a class $\epsilon \in H^1(M,\mathbb{Z}_2)$. Let $M_\epsilon \to M$ be the double cover and $\sigma$ the corresponding involution. Then a $d_\nabla$-closed $L$-valued $2$-form $\omega$ on $M$ is the same as a closed $2$-form $\omega'$ on $M_\epsilon$ such that $\sigma^* \omega' = -\omega'$. If $\omega'$ is non-degenerate then we have a symplectic structure on $M_\epsilon$ such that $\sigma$ is an anti-symplectic involution.
\end{example}

\begin{example} Suppose $D$ is a graph of an $L^{-1}$-valued bivector $\beta \in \Gamma( L^{-1} \otimes \wedge^2 TM)$, that is $D = \{ \beta(\xi) + \xi \, | \, \xi \in L \otimes T^*M \}$. From $L$ we get a graded algebra $\mathcal{A} = \bigoplus_{i=0}^\infty \mathcal{C}^\infty(M,(L^*)^i )$. Notice that this is the algebra of functions on the total space of $L$ which are polynomial in the fibres. Define a bracket $\{ \, , \, \} : \mathcal{A} \otimes \mathcal{A} \to \mathcal{A}$ by $\{f,g\} = \beta( \nabla f , \nabla g)$ and observe that $\{ \, , \, \}$ has degree $1$. Since integrability is a local condition it follows from the case where $L$ is trivial that $D$ is integrable if and only if $\{ \, , \, \}$ satisfies the Jacobi identity and thus makes $\mathcal{A}$ into a Poisson algebra. We see that $\beta$ determines a Poisson structure on the total space of $L$ which is linear in the fibres.
\end{example}


\subsection{Extension to orientifolds}\label{eto}

The case of an exact conformal Courant algebroid where the flat line bundle $L$ is orthogonal admits a generalization relevant to orientifolds. First note that if $L$ is a flat orthogonal line bundle on $M$ then $L$ corresponds to a class $\epsilon \in H^1(M,\mathbb{Z}_2)$ and determines a double cover $M_\epsilon \to M$ and involution $\sigma : M_\epsilon \to M_\epsilon$. A closed $L$-valued $3$-form on $M$ amounts to a closed $3$-form $H \in \Omega^3(M_\epsilon)$ on $M_\epsilon$ such that $\sigma^* H = -H$. An immediate generalization is to consider a manifold $X$ with involution $\sigma : X \to X$ without assuming that $\sigma$ is fixed-point free. Suppose we have a closed $3$-form $H \in \Omega^3(X)$ such that $\sigma^* H = -H$. We now consider the Courant algebroid $E = TX \oplus T^*X$ equipped with the $H$-twisted Dorfman bracket. We can lift $\sigma$ to an involution $\tilde{\sigma} : E \to E$ given by $\tilde{\sigma}( X , \xi ) = (\sigma_* X , - \sigma^* \xi )$. We observe that $\tilde{\sigma}$ preserves the $H$-twisted Dorfman bracket on $E$ but only respects the natural pairing $\langle \, , \, \rangle$ up to sign. The space of sections $A = \Gamma(E)^{\tilde{\sigma}}$ of $E$ invariant under $\tilde{\sigma}$ has the structure of a Leibniz algebra\footnote{A {\em Leibniz algebra} is a vector space $A$ with bilinear map $[ \, , \, ] : A \otimes A \to A$ satisfying the Leibniz identity $[a,[b,c]] = [[a,b],c] + [b,[a,c]]$.} with bracket given by the $H$-twisted Dorfman bracket. If $\sigma$ has fixed points then $A$ is not the space of sections of a vector bundle on $X/\sigma$, so $A$ is not a Leibniz algebroid, but it is at least a sheaf of Leibniz algebras. We can regard $A$ as the natural generalization of conformal Courant algebroids to the case of general involutions.\\

A slightly more general contruction applies to orientifolds. By an {\em orientifold}\footnote{This is really only the definition for a global quotient orientifold. More general definitions are possible.} we mean a global quotient groupoid $X // G$ where $X$ is a smooth manifold and $G$ is a discrete group which acts faithfully and properly on $X$. This is distinguished from orbifolds by also specifying a homomorphism $\phi : G \to \mathbb{Z}_2$. In the string theory elements of $g$ such that $\phi(g) = 1 \in \mathbb{Z}/2\mathbb{Z} \simeq \mathbb{Z}_2$ are accompanied by the world-sheet parity reversal operator $\Omega$. Suppose that $H$ is a closed $3$-form on $X$ such that for all $g \in G$ we have $g^* H = (-1)^{\phi(g)} H$. We give $E = TX \oplus T^*X$ the Courant algebroid structure with $H$-twisted Dorfman bracket. Next each $g \in G$ lifts to a bundle morphism $\tilde{g} : E \to E$ given by $\tilde{g}(X,\xi) = (g_* X , (-1)^{\phi(g)}(g^{-1})^* \xi)$. It is clear that the space $A = \Gamma(E)^G$ of $G$-invariant sections of $E$ form a Leibniz algebra which we regard as the orientifold generalization of exact Conformal Courant algebroids. We claim that the structure of such Leibniz algebras can be used to describe T-duality of orientifolds. In Section \ref{tdot} we pursue this claim for the case where $G = \mathbb{Z}_2$, $\phi$ is the identity and the involution acts freely so that $X = M_\epsilon$ is a double cover $M_\epsilon \to M$ for $M = X/\mathbb{Z}_2$ and $\epsilon$ is some element of $H^1(M,\mathbb{Z}_2)$.


\section{Topological T-duality with orientifold twists}\label{tdot}


\subsection{Topological T-duality}\label{ttd}

T-duality is a duality between spaces $X,\hat{X}$ which are torus bundles over a common base. The spaces $X,\hat{X}$ are equipped with some additional structure $h,\hat{h}$ which we refer to as {\em fluxes}. In the simplest form the fluxes are gerbes and roughly speaking the duality between $(X,h)$ and $(\hat{X},\hat{h})$ is a relation between the Chern class on one side and the flux on the other side \cite{bem},\cite{bunksch},\cite{bar}. The T-duality relation can be neatly captured as an isomorphism of Courant algebroids associated to $(X,h)$ and $(\hat{X},\hat{h})$ \cite[Chapter 8]{gual},\cite[Chapter 7]{cav}. We will consider extending this T-duality isomorphism to the case of conformal Courant algebroids. Demanding an isomorphism of conformal Courant algebroids leads us to formulate a definition of topological T-duality in this more general setting. 

We have already described in Section \ref{eto} how the structure of Conformal Courant algebroids leads naturally towards orientifolds. Our T-duality isomorphism for conformal Courant algebroids can be understood as a first step towards T-duality for orientifolds. Thus we are considering T-duality for orientifolds of the particularly simple form $X = X_\epsilon / / \mathbb{Z}_2$, that is we have a manifold $X$ which is a circle bundle $\pi : X \to M$ over a base $M$ and a class $\epsilon \in H^1(M,\mathbb{Z}_2)$. We then take $X_\epsilon$ to be the associated double cover with natural $\mathbb{Z}_2$-action and take the groupoid $X_\epsilon // \mathbb{Z}_2$ as our orientifold. Despite the simplicity of this setup this already allows us to extend topological T-duality into new territory and leads us to propose a T-duality isomorphism in twisted $KR$-theory. For simplicity we choose to restrict to the case of T-duality for circle bundles but it is certainly possible to extend this to higher rank torus bundles.\\

To express T-duality we first introduce the notion of an orientifold background. The idea is that an orientifold background captures the topological data required for an orientifold string theory background, at least in the simple case where the orientifold is defined by a double cover $\epsilon \in H^1(X,\mathbb{Z}_2)$.

\begin{definition}
An {\em orientifold background} $(X,\epsilon,\mathcal{G})$ consists of a smooth manifold $X$, a class $\epsilon \in H^1(X,\mathbb{Z}_2)$ called the {\em orientifold class} and a graded $\epsilon$-twisted gerbe $\mathcal{G}$ on $X$. Alternatively if $\mathcal{G}$ is classified by the pair $(\alpha,h) \in H^1(X,\mathbb{Z}_2) \times H^3(X,\mathbb{Z}_\epsilon)$ we will equally consider $(X,\epsilon,(\alpha,h))$ to define an orientifold background, even though $(\alpha,h)$ only determines a gerbe $\mathcal{G}$ up to isomorphism.
\end{definition}

The T-duality relation involves orientifold backgrounds which are torus bundles. For simplicity we will consider only the case of circle bundles and capture this structure in the notion of a T-duality background.
\begin{definition}
An orientifold background $(X,\epsilon',(\alpha',h))$ is said to be a {\em T-duality background} if $X$ is a circle bundle $\pi : X \to M$ over a smooth manifold $M$ and in addition there are classes $\epsilon,\alpha \in H^1(M,\mathbb{Z}_2)$ such that $\epsilon' = \pi^*(\epsilon)$, $\alpha' = \pi^*(\alpha)$. We write $(\pi : X \to M, \epsilon , (a , h))$ or simply $(X,\pi,\epsilon,(\alpha,h))$ to denote such a T-duality background.
\end{definition}

We could extend this definition to higher rank torus bundles. Extending T-duality in the higher rank setting is conceptually straightforward though there are some additional difficulties. We will address the higher rank case for $\epsilon = 0$ in \cite{bar1}.\\

Note that since the fibres of a circle bundle $\pi : X \to M$ are connected the pull-back map $\pi^* : H^1(M,\mathbb{Z}_2) \to H^1(X,\mathbb{Z}_2)$ is injective so if classes $\epsilon,\alpha$ as in the above definition exist, they are uniquely determined by $\epsilon',\alpha'$. We remark that there is a natural notion of isomorphism of T-duality backgrounds. For the purposes of T-duality we will consider the base $M$ to be fixed so that two T-duality backgrounds $(\pi : X \to M,\epsilon,(\alpha,h))$ and $(\pi' : X' \to M, \epsilon',(\alpha',h'))$ are isomorphic if there is a bundle isomorphism $\phi : X \to X'$ covering the identity such that $\phi^*(h') = h$ and in addition $\epsilon' = \epsilon, \alpha' = \alpha$.\\

Before we can formulate a definition of topological T-duality for T-duality backgrounds we need first to recall the classification of circle bundles which are not necessarily oriented \cite{bar}. We recall that circle bundles on $M$ are in bijection with principal ${\rm O}(2)$-bundles. Given a principal ${\rm O}(2)$-bundle $P \to M$ over $M$ we can form the associated circle bundle $X = P \times_{{\rm O}(2)} S^1$ by having ${\rm O}(2)$ act on the circle in the standard way. Alternatively $X$ is the quotient $P/\mathbb{Z}_2$ where the action of $\mathbb{Z}_2$ is through a subgroup $\mathbb{Z}_2 \to {\rm O}(2)$ generated by a reflection. It is not hard to see that every circle bundle over $M$ arises in this way. Suppose that $\pi : X \to M$ is a circle bundle over $M$. The bundle of fibre orientations of $X$ determines a double cover of $M$ and thus a class $\xi = w_1(X) \in H^1(M,\mathbb{Z}_2)$ which we call the {\em Stiefel-Whitney class} of $X$. If $M$ and $X$ are smooth, which is the case of interest to us then $\xi$ is the first Stiefel-Whitney class of a line bundle $V$ on $M$ with the property that $\pi^*(V)$ is isomorphic to the vertical bundle ${\rm Ker}(\pi_*) : TX \to TM$. The double cover $p_\xi : X_\xi \to X$ determined by $\xi$ can be given the structure of a principal ${\rm O}(2)$-bundle over $M$ and we recover $X$ as the $\mathbb{Z}_2$-quotient $X = X_\xi/\mathbb{Z}_2$.\\

In addition to the Stiefel-Whitney class $\xi = w_1(X)$ there is another cohomology class $c = c_1(X) \in H^2(M,\mathbb{Z}_\xi)$ called the {\em twisted Chern class} \cite{bar}, which takes values in the local system $\mathbb{Z}_\xi$ with $\mathbb{Z}$-coefficients determined by $\xi$. Strictly speaking the twisted Chern class takes values in a local system $S$ which is isomorphic to $\mathbb{Z}_\xi$, but not canonically. There is not a unique isomorphism $S \to \mathbb{Z}_\xi$ because $\mathbb{Z}_\xi$ has a non-trivial automorphism which acts by negation. Thus the twisted Chern class as an element of $H^2(M,\mathbb{Z}_\xi)$ is really only defined up to an overall sign. This isn't really a problem because one finds that the isomorphism classes of circle bundles on $M$ are in bijection with pairs $(\xi,c)$ where $\xi \in H^1(M,\mathbb{Z}_2)$, $c \in H^2(M,\mathbb{Z}_\xi)$ modulo the equivalence $(\xi,c) \sim (\xi,-c)$ \cite{bar}.\\

For a circle bundle $\pi : X \to M$ with Stiefel-Whitney class $\xi$ and twisted Chern class $c$ there is an exact sequence relating the cohomology of $X$ and $M$, namely the {\em Gysin sequence} \cite{bar}
\begin{eqnarray}
\cdots \to H^k(M,G) \buildrel \pi^* \over \to H^k(X,G) \buildrel \pi_* \over \to & \label{gys} \\ 
H^{k-1}(M,G \otimes \mathbb{Z}_\xi) & \buildrel \smallsmile c \over \to H^{k+1}(M,G) \to \cdots \nonumber
\end{eqnarray}
where $G$ is an abelian group or more generally a local system of abelian groups on the base $M$. The map $\pi_*$ is called the {\em Gysin homomorphism} or push-forward map and can be understood as a kind of fibre integration. One way to define $\pi_*$ is through the Leray-Serre spectral sequence for $\pi : X \to M$.\\

We are almost ready to state our definition of T-duality for T-duality backgrounds. The construction of an isomorphism between conformal Courant algebroids and twisted cohomology taken up in Section \ref{otd} is the underlying motivation for the present definition. If $\pi : X \to M$, $\hat{\pi} : \hat{X} \to M$ are circle bundles over a common base $M$, the fibre product $C = X \times_M \hat{C}$ will be called the {\em correspondence space} and we will visualize the relation between the spaces $X,\hat{X}$ by the following commuative diagram:
\begin{equation*}\xymatrix{
& C \ar[dl]_p \ar[dr]^{\hat{p}} \ar[dd]^{q} & & \\
X \ar[dr]_\pi & & \hat{X} \ar[dl]^{\hat{\pi}} \\
& M & &
}
\end{equation*}

Later we will need to make use of the Leray-Serre spectral sequences associated to these bundles. For clarification we denote by $E_r^{p,q}(f,A)$ the Leray-Serre spectral sequence associated to a fibre bundle $f : Y \to Z$ with coefficients in $A$ which is an abelian group or local system of abelian groups on the base. In \cite{bar} the $E_2$-stage and differential $d_2$ were determined for torus bundles with structure group ${\rm Aff}(T^n) = {\rm GL}(n,\mathbb{Z}) \ltimes T^n$, the group of affine transformations of the $n$-torus. If $X,\hat{X} \to M$ are circle bundles then $X,\hat{X},C = X \times_M \hat{X}$ are affine, so these results are applicable here. We recall the relevant details in Appendix \ref{aftbss}

\begin{definition}\label{tdd}
Let $(\pi : X \to M , \epsilon , (\alpha , h))$, $(\hat{\pi} : \hat{X} \to M , \hat{\epsilon} , (\hat{\alpha},\hat{h}))$ be two T-duality backgrounds over a common base $M$. Let $\xi,\hat{\xi}$ be the Stiefel-Whitney classes and $c,\hat{c}$ the twisted Chern classes for $X$ and $\hat{X}$ respectively. We say that $(\hat{X},\hat{\pi},\hat{\epsilon},(\hat{\alpha},\hat{h}))$ is {\em T-dual} to $(X,\pi,\epsilon,(\alpha,h))$ if the following conditions hold:
\begin{itemize}
\item[(T1)]{$\hat{\epsilon} = \epsilon$,}
\item[(T2)]{$\hat{\xi} = \xi + \epsilon$,}
\item[(T3)]{$\hat{\alpha} = \alpha + \xi$,}
\item[(T4)]{$\pi_*(h) = \hat{c}$, $\hat{\pi}_*(\hat{h}) = c$,}
\item[(T5)]{$p^*h = \hat{p}^*\hat{h}$.}
\end{itemize}
\end{definition}

This definition is dictated by the definition of topological T-duality in \cite{bem} (the case $\epsilon = \alpha = \xi = 0$), its extension to general circles in \cite{bar} (the case $\epsilon = 0$) and the requirement of an isomorphism of conformal Courant algebroids. Note that the T-duality relation as stated is not quite symmetric. Indeed if $(\hat{X},\hat{\pi},\hat{\epsilon},(\hat{\alpha},\hat{h}))$ is T-dual to $(X,\pi,\epsilon,(\alpha,h))$ then a T-dual of $(\hat{X},\hat{\pi},\hat{\epsilon},(\hat{\alpha},\hat{h}))$ is given by $(X,\pi,\epsilon,(\alpha+\epsilon,h))$. However we can state T-duality in a slightly different way that does make it symmetric. We will see that T-duality entails an isomorphism of twisted cohomology $H_\epsilon^{t,\alpha,h}(X) \simeq H_{\hat{\epsilon}}^{t-1,\hat{\alpha},\hat{h}}(\hat{X})$. Introduce an integer $t$ so that a T-duality background now consists of a tuple $(X,\pi,\epsilon,(t,\alpha,h))$ with $t \in \mathbb{Z}$. The T-duality relation as stated in Definition \ref{tdd} can be supplemented by the condition $\hat{t} = t-1$. If we perform T-duality twice we go from $(X,\pi,\epsilon,(t,\alpha,h))$ to $(X,\pi,\epsilon,(t-2,\alpha + \epsilon,h))$. But now recall that there are canonical isomorphisms $H_\epsilon^{t,\alpha,h}(X) \simeq H_\epsilon^{t-2,\alpha+\epsilon,h}(X)$. This suggests that we should regard the T-duality backgrounds $(X,\pi,\epsilon,(t,\alpha,h))$ and $(X,\pi,\epsilon,(t-2,\alpha+\epsilon,h))$ as equivalent. We have that modulo equivalence T-duality is symmetric. Note that the integer $t$ taken modulo $2$ determines whether the orientifold background is being used for type IIB or IIA orientifold string theory.\\

\begin{proposition}\label{tdprop} Every T-duality background has a T-dual which is unique up to isomorphism.
\begin{proof}
Let $(X,\pi,\epsilon,(\alpha,h))$ be a T-duality background where $X$ has Stiefel-Whitney class $\xi$ and twisted Chern class $c \in H^2(M,\mathbb{Z}_\xi)$. Following (T2) and (T4) of Definition \ref{tdd} we define $\hat{\xi} = \xi + \epsilon$ and $\hat{c} = \pi_*(h)$. The pair $(\hat{\xi},\hat{c})$ determines a circle bundle $\hat{\pi} : X \to M$ with Steifel-Whitney class $\hat{\xi}$ and twisted Chern class $\hat{c}$. Following Definition \ref{tdd}, (T1),(T3) we further define $\hat{\epsilon} = \epsilon$, $\hat{\alpha} = \alpha + \xi$. From the Gysin sequence (\ref{gys}) applied to $X$ we find that $c \smallsmile \hat{c} = 0$ and further applying the Gysin sequence to $\hat{X}$ we find that there exists a class $h' \in H^3(\hat{X},\mathbb{Z}_{\epsilon})$ such that $\hat{\pi}_* h' = c$. So far we have constructed a second T-duality background $(\hat{X},\hat{\pi},\hat{\epsilon},(\hat{\alpha},h'))$ which satisfies conditions (T1)-(T4) of Definition \ref{tdd}. The only difficulty is that condition (T5) need not be satisfied. Let $d = \hat{p}^*(h') - \hat{p}^*(h)$.

\begin{lemma}\label{exist}
There exists $a \in H^3(M,\mathbb{Z}_\epsilon)$ such that $p^*(h) - \hat{p}^*(h') = q^*(a)$.
\end{lemma}
We omit the proof of this result as it is almost identical to the proof in \cite{bar}. Given Lemma \ref{exist} we then set $\hat{h} = h' + \hat{\pi}^*(a)$. It follows immediately that $(\hat{X},\hat{\pi},\hat{\epsilon},(\hat{\alpha},h'))$ is a T-dual to $(X,\pi,\epsilon,(\alpha,h))$.\\

For uniqueness note that by Definition (\ref{tdd}) any other T-dual is up to isomorphism of the form $(\hat{X},\hat{\pi},\hat{\epsilon},(\hat{\alpha},k))$ for some $k \in H^3(\hat{X},\mathbb{Z}_\epsilon)$ such that $\hat{\pi}_*(k) = c$, $\hat{p}^*(k) = p^*(h)$. Let $\mu = k - \hat{h}$. We see immediately that $\hat{\pi}_*(\mu) = 0$ and $\hat{p}^*(\mu) = 0$.

\begin{lemma}\label{mainlem}
Let $\mu \in H^3(\hat{X} , \mathbb{Z}_{\epsilon})$ be such that $\hat{\pi}_* \mu = 0$ and $\hat{p}^* \mu = 0$. There exists $e \in H^1(M,\mathbb{Z}_{\hat{\xi}})$ such that $\mu = \hat{\pi}^*( c \smallsmile e)$.
\end{lemma}

\begin{lemma}\label{gauge}
Let $\pi : X \to M$ be a circle bundle with fiber orientation class $\xi \in H^1(M,\mathbb{Z}_2)$. For any $\alpha \in H^1(M,\mathbb{Z}_\xi)$ there exists a bundle isomorphism $\phi : X \to X$ such that for any $x \in H^k(E,\mathbb{Z}_\epsilon)$ we have
\begin{equation*}
\phi^* x = x + \pi^*( \alpha \smallsmile \pi_* x ).
\end{equation*}
\end{lemma}

The proof of Lemma \ref{mainlem} is given below. The proof of Lemma \ref{gauge} is omitted since it is identical to the proof \cite{bar}.\\

By Lemmas \ref{mainlem} and \ref{gauge} we can find $\alpha \in H^1(M,\mathbb{Z}_{\hat{\xi}})$ such that $k = \hat{h} + \hat{\pi}^*( \alpha \smallsmile c)$ and a bundle isomorphism $\phi : \hat{X} \to \hat{X}$ corresponding to $\alpha$ such that
\begin{equation*}
\phi^* \hat{h} = \hat{h} + \hat{\pi}^*( \alpha \smallsmile \hat{\pi}_* \hat{h}) = \hat{h} + \hat{\pi}^*( \alpha \smallsmile c) = k.
\end{equation*}
This establishes uniqueness.
\end{proof}
\end{proposition}


\begin{proof}[Proof of Lemma \ref{mainlem}]

Since $\hat{\pi}_* \mu = 0$ we have from the Gysin sequence that there exists $\nu \in H^3(M,\mathbb{Z}_\epsilon)$ such that $\hat{\pi}^*\nu = \mu$. It follows also that $q^*(\nu) = 0$. Let $\nu_3$ denote the image of $\nu$ in $E_3^{3,0}(q,\mathbb{Z}_\epsilon)$. Now since $q^*(\nu) = 0$ it follows that there exists $r \in E_3^{0,2}(q,\mathbb{Z}_\epsilon)$ such that $\nu_3 = d_3(r)$. Next observe that $E_3^{0,2}(q,\mathbb{Z}_\epsilon)$ is the subgroup of elements $r \in E_2^{0,2}(q,\mathbb{Z}_\epsilon) = H^0(M,\mathbb{Z}_{\epsilon + \xi + \hat{\xi}}) = H^0(M,\mathbb{Z}) = \mathbb{Z}$ such that $rc = 0$, $r\hat{c} = 0$.\\

Suppose first that $c,\hat{c}$ are not both torsion. Then $E_3^{0,2}(q,\mathbb{Z}_\epsilon) = 0$ and $\nu_3 = 0$. Therefore there exists $(\alpha,\hat{\alpha}) \in H^1(M,\mathbb{Z}_{\hat{\xi}}) \oplus H^1(M,\mathbb{Z}_\xi)$ such that $\nu = d_2(\alpha,\hat{\alpha}) = c \smallsmile \alpha + \hat{c} \smallsmile \hat{\alpha}$. Now if we apply $\hat{\pi}^*$ and use the fact that $\hat{\pi}^*(\hat{c}) = 0$ then we obtain $\mu = \hat{\pi}^* \nu = \hat{\pi}^*( c \smallsmile \alpha)$ as required.\\

We assume now that $c,\hat{c}$ are both torsion. Let $m$ be the least positive integer such that $mc = m\hat{c} = 0$. Then $E_3^{0,2}(q,\mathbb{Z}_\epsilon) \subseteq E_2^{0,2}(q,\mathbb{Z}_\epsilon) = H^0(M,\mathbb{Z})$ is the subgroup $m\mathbb{Z} \subseteq \mathbb{Z}$. We have that $\nu_3 = d_3(r)$ where $r \in m\mathbb{Z}$. Consider the short exact sequence $0 \to \mathbb{Z}_\xi \buildrel m \over \to \mathbb{Z}_\xi \to (\mathbb{Z}_m)_\xi \to 0$ where $(\mathbb{Z}_m)_\xi$ is the local system with coefficients $\mathbb{Z}_m$ obtained from twisting by $\xi$. As $mc = 0$ we have that there exists $a \in H^1(M,(\mathbb{Z}_m)_\xi)$ such that $c = \delta(a)$, where $\delta$ is the coboundary operator associated to the above short exact sequence. Similarly we have $\hat{a} \in H^1(M,(\mathbb{Z}_m)_{\hat{\xi}})$ such that $\hat{c} = \delta(\hat{a})$. The element $a$ can be thought of as an extension of the homomorphism $\pi_1(M) \to \mathbb{Z}_2$ determined by $\xi$ to a homomorphism $\pi_1(M) \to D_{m}$ where $D_{m} = \mathbb{Z}_2 \ltimes \mathbb{Z}_m$ is the dihedral group of order $2m$. It follows that the pair $(\xi,a)$ correspond to a pair of maps $(f,g)$ forming a commutative diagram as follows
\begin{equation*}\xymatrix{
M \ar[r]^-g \ar[dr]_-f & K(D_{m},1) \ar[d]^-{pr} \\
& K(\mathbb{Z}_2,1)
}
\end{equation*}
where $pr : K(D_{m},1) \to K(\mathbb{Z}_2,1)$ is the map associated to the homomorphism $D_{m} \to \mathbb{Z}_2$ obtained by projecting out $\mathbb{Z}_m$. There exists a universal class $\xi_0 \in H^1(K(\mathbb{Z}_2,1),\mathbb{Z}_2)$ such that $\xi = f^*(\xi_0)$ and a universal class $a_0 \in H^1( K(D_{m},1) , (\mathbb{Z}_m)_{pr^*(\xi_0)})$ such that $g^*(a_0) = a$. If we let $c_0 = \delta(a_0) \in H^2(K(D_{m},1),\mathbb{Z}_{pr^*(\xi_0)})$ then $c = g^*(c_0)$ as well. Similarly we have a pair of maps $\hat{f} : M \to K(D_{m},1)$, $\hat{g} : M \to K(D_{m},1)$ such that $\hat{f}^*(\xi_0) = \hat{\xi}$, $pr \circ \hat{g} = \hat{f}$ and $\hat{a} = \hat{g}^*(a_0)$, so $\hat{c} = \hat{g}^*(c_0)$.\\

Let $X_0 \to K(D_{m},1)$ be the circle bundle with monodromy $pr^*(\xi_0)$ and twisted Chern class $c_0$. We can identify $X$ as the pull-back of $X_0$ under $g$ and $\hat{X}$ the pull-back of $X_0$ under $\hat{g}$. Next let $Y = K(D_{m},1) \times K(D_{m},1)$ and let $pr_1,pr_2$ be the projections of $Y$ to the respective factors. Set $C_0 = pr_1^*(X_0) \times_Y pr_2^*(X_0)$, so $C = X \times_M \hat{X}$ is the pull-back of $C_0$ under $(g,\hat{g})$. Let us define $\xi_i = pr_i^* pr^* (\xi_0)$, $c_i = pr_i^*(c_0)$ and set $\epsilon_0 = \xi_1 + \xi_2$. We have that $\xi = (g,\hat{g})^* \xi_1 = (g,\hat{g})^* \xi_2$, $(g,\hat{g})^* \epsilon_0 = \epsilon$, $c = (g,\hat{g})^* c_1$ and $\hat{c} = (g,\hat{g})^* c_2$.\\

Let $q_0$ be the projection $q_0 : C_0 \to Y$. We have $E_2^{0,2}(q_0,\mathbb{Z}_{\epsilon_0}) = \linebreak H^0(Y , \mathbb{Z}_{\epsilon_0 + \xi_1 + \xi_2}) = H^0(Y,\mathbb{Z}) = \mathbb{Z}$ and since $mc_1 = 0$, $mc_2 = 0$ we find that $E_3^{0,2}(q_0,\mathbb{Z}_{\epsilon_0}) \subseteq E_2^{0,2}(q_0,\mathbb{Z}_{\epsilon_0})$ corresponds to $m\mathbb{Z} \subseteq \mathbb{Z}$, so in particular the integer $r$ can be thought as $r \in E_3^{0,2}(q_0,\mathbb{Z}_{\epsilon_0})$. Let $\psi \in E_2^{3,0}(q_0,\mathbb{Z}_{\epsilon_0}) = H^3(Y,\mathbb{Z}_{\epsilon_0})$ be a representative for $d_3(r) \in E_3^{3,0}(q_0,\mathbb{Z}_\epsilon)$. The point is that it follows that $(g,\hat{g})^* \psi \in H^3(M,\mathbb{Z}_{(g,\hat{g})^*{\epsilon_0}}) = H^3(M,\mathbb{Z}_\epsilon)$ is a representative for $d_3(r) \in E_3^{3,0}(q,\mathbb{Z}_\epsilon)$. Therefore there exists $(e,\hat{e}) \in E_2^{1,1}(q,\mathbb{Z}_\epsilon) = H^1(M,\mathbb{Z}_{\hat{\xi}}) \oplus H^1(M,\mathbb{Z}_\xi)$ such that $\nu = (g,\hat{g})^*\psi + d_2(e,\hat{e}) = (g,\hat{g})^*\psi + c \smallsmile e + \hat{c} \smallsmile \hat{e}$. If we apply $\hat{\pi}$ we obtain 
\begin{equation}\label{formmu}
\mu = \hat{\pi}^*\nu = \hat{\pi}^*(g,\hat{g})^* \psi + \hat{\pi}^*(c \smallsmile e)
\end{equation}
where we have used $\hat{\pi}^*(\hat{c}) = 0$. Observe now that $\hat{\pi}^*(g,\hat{g})^* = (g \circ \hat{\pi} , \hat{g} \circ \hat{\pi})^*$. Note that since $\hat{\pi}^* \hat{c} = 0$ we have $(g \circ \hat{\pi})^* c_0 = 0$. This is precisely the condition that $g \circ \hat{\pi} : \hat{X} \to K(D_{m},1)$ factors through $K(\mathbb{Z}_2,1)$, that is there is a map $h : \hat{X} \to K(\mathbb{Z}_2,1)$ such that $g \circ \hat{\pi}$ is homotopic to the composition $\hat{X} \buildrel h \over \to K(\mathbb{Z}_2,1) \buildrel i \over \to K(D_{m},1)$ where $i : K(\mathbb{Z}_2,1) \to K(D_{m},1)$ is the natural map corresponding to the homomorphism $\mathbb{Z}_2 \to D_{m}$. One sees that $i^* pr^* \xi_0 = \xi_0$ and it follows that $h^* \xi_0 = \xi$. Thus $h$ is homotopic to $\hat{f} \circ \hat{\pi} : \hat{X} \to K(\mathbb{Z}_2,1)$. The main point is that $\hat{g} \circ \hat{\pi}$ is homotopic to $i \circ \hat{f} \circ \hat{\pi}$, so $(g \circ \hat{\pi} , \hat{g} \circ \hat{\pi})^* \psi  = (g \circ \hat{\pi} , \hat{f} \circ \hat{\pi})^* (id , i)^* \psi$, where $(id,i)$ is the map $(id,i) : K(D_{m},1) \times K(\mathbb{Z}_2,1) \to K(D_{m},1) \times K(D_{m},1)$.

Note that $i^*(X_0)$ has trivial twisted Chern class so admits a section $\sigma : K(\mathbb{Z}_2 , 1) \to i^*(X_0)$. We then get a commutative diagram
\begin{equation*}\xymatrix{
pr_1^*(X_0) \ar[d]^{q_1} \ar[r]^-{id \times \sigma} & C_0 \ar[d]^-{q_0} \\
K(D_{m},1) \times K(\mathbb{Z}_2,1) \ar[r]^-{(id,i)} & Y
}
\end{equation*}
where we let $pr_1 , pr_2$ denote the projections of $K(D_{m},1) \times K(\mathbb{Z}_2,1)$ to the first and second factors. Let $E_r^{p,q}(q_1,\mathbb{Z}_{\epsilon_1})$ be the Leray-Serre spectral sequence for $q_1 : pr_1^*(X_0) \to K(D_{m},1) \times K(\mathbb{Z}_2,1)$ with coefficients in $\mathbb{Z}_{\epsilon_1}$ where $\epsilon_1 = (id,i)^*\epsilon_0 = pr_1^* pr^*(\xi_0) + pr_2^*(\xi_0)$. We have that $(id,i)^* \psi $ is a representative in $E_2^{3,0}(q_1,\mathbb{Z}_{\epsilon_1})$ for $(id,i)^*d_3(r) = d_3( (id,i)^* r) \in E_3^{3,0}(q_1,\mathbb{Z}_{\epsilon_1})$, but note that $q_1 : pr_1^*(X_0) \to K(D_{m},1) \times K(\mathbb{Z}_2,1)$ is a circle bundle so in fact $(id,i)^* r \in E_3^{0,2}(q_1,\mathbb{Z}_{\epsilon_1}) = 0$ and thus $(id,i)^* \psi = d_2( j ) = pr_1^*(c_0) \smallsmile j$ for some $j \in H^1(K(D_{m}, 1) \times K(\mathbb{Z}_2,1) , \mathbb{Z}_{pr_2^* \xi_0})$. We now have $(g \circ \hat{\pi} , \hat{g} \circ \hat{\pi})^* \psi = (g \circ \hat{\pi} , \hat{f} \circ \hat{\pi})^* (id,i)^* \psi = \hat{\pi}^* (g,\hat{f})^* (pr_1^*(c_0) \smallsmile j) = \hat{\pi}^*( c \smallsmile k)$ where $k = (g,\hat{f})^* j \in H^1(M,\mathbb{Z}_{\hat{\xi}})$. Combining with (\ref{formmu}) we obtain $\mu = \hat{\pi}^*( c \smallsmile (k+e))$ which on setting $\alpha = k+e$ has the desired form.
\end{proof}


\subsection{T-duality and twisted spin structures}\label{tdtss}

In \cite{dfm0} it is proposed that the Ramond-Ramond fields for a string theory background $(X,\epsilon,(\mathcal{G}))$ have fluxes which are valued in $KR$-theory twisted by $\mathcal{G}$, denoted $KR^*(X,\mathcal{G})$. The fields themselves can be thought of as living in the corresponding twisted differential cohomology theory $\check{K}R^*(X,\check{\mathcal{G}})$, where $\check{\mathcal{G}}$ is now a twist in differential cohomology. However the Ramond-Ramond fields in string theory are supposed to satisfy a self-duality condition. In order the express self-duality condition \cite{dfm0} introduces the notion of a {\em twisted spin structure}. In particular there are topological obstructions to existence of twisted spin structures. Although we will do not consider self-duality here we will show that the topological obstructions for a twisted spin structure are compatible with our definition of T-duality. This provides a consistency check that our formulation of T-duality is physically correct.\\

Let $W$ be an orthogonal vector bundle of rank $n$ on $X$ and suppose we also have classes $\epsilon,a \in H^1(X,\mathbb{Z}_2)$. The data $W,\epsilon,a$ determines a principal ${\rm O}(n) \times \mathbb{Z}_2 \times \mathbb{Z}_2$-bundle $P_{W,\epsilon,a} \to X$, where the two factors of $\mathbb{Z}_2$ come from the double covers associated to $\epsilon,a$. In \cite{dfm} a {\em twisted spin-structure} on $W$ is defined as the combination of a reduction of structure group of $P_{W,\epsilon,a}$ from ${\rm O}(n) \times \mathbb{Z}_2 \times \mathbb{Z}_2$ to a certain subgroup $G'_i$ of index $2$, followed by a lift of the structure group to a certain $2$-fold cover $G_i \to G'_i$. There are actually two choices for the groups $G_i , G'_i$ indicated by the index $i \in \mathbb{Z}_2$. The reason for this is that the $i=0$ version is supposed to correspond to type IIB theory and $i=1$ to type IIA. Let $\tilde{{\rm O}}(n) = {\rm O}(n) \times \mathbb{Z}_2 \times \mathbb{Z}_2$. Then $G'_0$ is defined to be the subgroup ${\rm SO}(n) \times \mathbb{Z}_2 \times \mathbb{Z}_2$ while $G'_1$ is defined to be the kernel of $\tilde{{\rm O}}(n) \to \mathbb{Z}_2 = \mathbb{Z}/2\mathbb{Z}$ which sends $(g,a,b)$ to $c + a$ where ${\rm det}(g) = (-1)^c$. Let $D_4$ denote the dihedral group of order $8$. Note that the center of $D_4$ is $\mathbb{Z}_2$ and that the quotient by the center is $\mathbb{Z}_2 \times \mathbb{Z}_2$. Thus we obtain a double cover $({\rm Pin}_-(n) \times D_4) / \{ \pm 1 \} \to \tilde{{\rm O}}(n)$ where $-1$ means the product of the kernel ${\rm Pin}_-(n) \to {\rm O}(n)$ with the central element of $D_4$. For $i=0,1$ we define $G_i$ to be the inverse image of $G'_i$ under the homomorphism $({\rm Pin}_-(n) \times D_4) / \{\pm 1\} \to \tilde{{\rm O}}(n)$.\\

The obstructions for $W$ to admit a twisted spin structure consist of classes $O_1 \in H^1(X,\mathbb{Z}_2)$ and $O_2 \in H^2(X,\mathbb{Z}_2)$. The reduction from $\tilde{{\rm O}}(n)$ to $G'_i$ is possible if and only if $O_1$ vanishes. A lift of structure group from $G'_i$ to $G_i$ is then possible if and only if $O_2$ vanishes as well. The obstructions are as follows
\begin{eqnarray}
O_1 &=& w_1(W) + i \epsilon \label{o1}, \\
O_2 &=& w_2(W) + i \epsilon^2 + a \epsilon. \label{o2}
\end{eqnarray}

In the context of a T-duality background $(\pi:X \to M,\epsilon,(t,\alpha,h))$ (complete with $t \in \mathbb{Z}$) we take $i$ to be $t$ mod $2$ and $a = \alpha + \frac{1}{2}t(t-1)\epsilon$. To explain this seemingly arbitrary change of variable recall that we can identify $(\pi:X \to M,\epsilon,(t,\alpha,h))$ and $(\pi:X \to M,\epsilon,(t+2,\alpha+\epsilon,h))$ as equivalent T-duality backgrounds. If we set $i = t$ mod $2$ and $a = \alpha + \frac{1}{2}t(t-1)$ then we find that the pair $(i,a)$ is the same for equivalent backgrounds and the set of all such pairs $(i,a) \in \mathbb{Z}_2 \times H^1(M,\mathbb{Z}_2)$ parametrize equivalence classes of backgrounds (for fixed $X \to M$, $\epsilon$ and $h$).\\

Let $V$ be the vertical bundle of $\pi : X \to M$ so that $TX$ is isomorphic to the pull-back under $\pi$ of the bundle $TM \oplus V$. The obstructions for a twisted spin structure for $TM \oplus V$ on $M$ are
\begin{eqnarray}
O_1 &=& w_1(TM \oplus V) + t \epsilon, \label{o3} \\
O_2 &=& w_2(TM \oplus V) + \frac{1}{2}t(t+1) \epsilon^2 + \alpha \epsilon. \label{o4}
\end{eqnarray}

\begin{proposition} Let $(\pi:X \to M,\epsilon,(t,\alpha,h))$ be a T-duality background and $(\hat{\pi}: \hat{X} \to M, \hat{\epsilon} , (\hat{t},\hat{\alpha},\hat{h}))$ a T-dual background. Let $O_i \in H^i(M,\mathbb{Z}_2)$ and $\hat{O}_i \in H^i(M,\mathbb{Z}_2)$ for $i=1,2$ be defined as in (\ref{o3}),(\ref{o4}). We have $\hat{O}_1 = O_1$ and $\hat{O}_2 = O_2 + \epsilon O_1$. Thus $O_1,O_2$ both vanish if and only if $\hat{O}_1,\hat{O}_2$ both vanish.
\begin{proof}
Let $\xi,\hat{\xi} \in H^1(M,\mathbb{Z}_2)$ be the first Stiefel-Whitney classes of $X,\hat{X}$. Recall the following T-duality relations $\hat{\epsilon} = \epsilon$, $\hat{\xi} = \xi + \epsilon$, $\hat{\alpha} = \alpha + \xi$, $\hat{t} = t-1$ which follow from the definition of $T$-duality. Next observe the following identities
\begin{eqnarray}
w_1(TM \oplus \hat{V} ) &=& w_1(TM \oplus V) + \epsilon \\
w_2(TM \oplus \hat{V} ) &=& w_2(TM \oplus V) + w_1(TM) \epsilon,
\end{eqnarray}
here $V,\hat{V}$ are the line bundles corresponding to $\xi, \hat{\xi}$. Combining these it is straightforward to show that $\hat{O}_1 = O_1$, $\hat{O}_2 = O_2 + \epsilon O_1$.
\end{proof}
\end{proposition}

\begin{remark}
This result can be extended to the case of higher rank torus bundles. The relevant T-duality relations in rank $n$ become\footnote{We derived these relations by demanding an isomorphism of conformal Courant algebroids. A simple consistency check is to observe that these relations are consistent with iterating T-duality on a space that is a fibre product of circle bundles.} $\hat{\epsilon} = \epsilon$, $\hat{\xi} = \xi + n\epsilon$, $\hat{\alpha} = \alpha + \xi$, $\hat{t} = t-n$ and one easily checks that $\hat{O}_1 = O_1$, $\hat{O}_2 = O_2 + n \epsilon O_1$. We conclude that our approach to T-duality is compatible with the notion of twisted spin structures.
\end{remark}


\section{T-duality of twisted cohomology and conformal Courant algebroids}\label{otd}

\subsection{Invariant forms and cohomology}\label{invco}

Let $\pi : X \to M$ be a circle bundle over $M$, $\xi = w_1(X) \in H^1(M,\mathbb{Z}_2)$ be the Stiefel-Whitney class and $c = c_1(X) \in H^2(M,\mathbb{Z}_\xi)$ the twisted Chern class. Let $\mathbb{Z}_\xi$ be the $\mathbb{Z}$-valued local system corresponding to $\xi$ and $V = \mathbb{Z}_\xi \otimes \mathbb{R}$ the corresponding flat orthogonal line bundle. As usual the double cover $ p_\xi : X_\xi \to X$ can be given the structure of a principal ${\rm O}(2)$-bundle $\pi_\xi : X_\xi \to M$, where $\pi_\xi = \pi \circ p_\xi$.

We say that a differential form $\omega$ on $X$ is {\em invariant} if $p_\xi^*(\omega)$ is ${\rm O}(2)$-invariant. If $L$ is a flat line bundle on $M$ then the pull-back $\pi_\xi^*(L)$ admits an action of ${\rm O}(2)$ and thus it makes sense to say that a form $\omega$ on $X$ with values in $\pi^*(L)$ is invariant since $p_\xi^*(\omega)$ has values in $p_\xi^* \pi^*(L) = \pi_\xi^*(L)$. 

Now suppose that $L,A$ are flat orthogonal line bundles on $M$ and define $S^{i,A}_L(X) = \bigoplus_{j \in \mathbb{Z}} A \otimes L^j \otimes \wedge^{i+2j} T^*X$ as in Section \ref{twco}. We let $\Gamma_{{\rm inv}}( S^{i,A}_L(X))$ denote the subspace of $\Gamma( S^{i,A}_L(X))$ consisting of invariant forms. Suppose that $H \in \Omega^3(X,L)$ is a closed $L$-valued $3$-form. Then as in Section \ref{twco} we have a differential $d_{\nabla,H} : \Gamma( S^{i,A}_L(X)) \to \Gamma( S^{i+1,A}_L(X))$ and we define the twisted cohomology groups $H_L^{i,H,A}(X)$ to be the cohomology of this complex. Suppose further that $H$ is invariant. In this case we observe that the differential $d_{\nabla,H} = d_\nabla + H \wedge $ preserves the subcomplex of invariant differential forms. Therefore we get a second set of cohomology groups $H^{i,H,A}_{L,{\rm inv}}(X)$ by taking cohomology of the complex of invariant forms. The inclusion $\iota : \Gamma_{{\rm inv}}( S^{i,A}_L(X)) \to \Gamma( S^{i,A}_L(X))$ of invariant forms induces a morphism $\iota_* : H^{i,H,A}_{L,{\rm inv}}(X) \to H^{i,H,A}_L(X)$. The proof used in \cite{bar} easily adapts to this situation giving:

\begin{proposition}\label{tciso}
The morphism $\iota_* : H^{i,H,A}_{L,{\rm inv}}(X) \to H^{i,H,A}_L(X)$ is an isomorphism.
\end{proposition}

Using a connection we can express invariant differential forms on $X$ in terms of differential forms on $M$. Let $A$ be a connection for the principal ${\rm O}(2)$-bundle $X_\xi \to M$. We can think of $A$ as a $V$-valued invariant $1$-form on $X$ such that the composition $V \to TX \buildrel A \over \to V$ is the identity. Here we think of $V$ as the vertical subbundle ${\rm Ker}(\pi_*)$ of $TX$. The curvature of $A$ is a $V$-valued $2$-form $F \in \Omega^2(M,V)$ such that $dA = \pi^*(F)$. The curvature is closed and thus defines a class $[F] \in H^2(M,V)$. It is not hard to see that $[F]$ is the image of the twisted Chern class under the change of coefficients $\mathbb{Z}_\xi \to \mathbb{Z}_\xi \otimes \mathbb{R} = V$.\\

Let $\omega \in \Gamma_{{\rm inv}}( S^{i,A}_L(X))$ be an invariant form. Using a connection $A$ we can uniquely decompose $\omega$ as $\omega = \pi^*(\omega_i) + \pi^*(\omega_{i-1}) \wedge A$ where $\omega_i \in \Gamma( S^{i,A}_L(M) )$ and $\omega_{i-1} \in \Gamma( S^{i-1,A \otimes V}_L(M))$. Therefore we have an isomorphism
\begin{equation}\label{decom}
\Gamma_{{\rm inv}}( S^{i,A}_L(X)) = \Gamma( S^{i,A}_L(M) ) \oplus \Gamma( S^{i-1,A \otimes V}_L(M)).
\end{equation}
Next let $H \in \Omega^3(X,L)$ be an invariant $L$-valued closed $3$-form. We can similarly decompose $H$ into $H = \pi^*(H_3) + \pi^*(H_2) \wedge A$ where $H_3 \in \Omega^3(M,L)$ and $H_2 \in \Omega^2(M,L \otimes V)$. Under the isomorphism (\ref{decom}) we find that the twisted differential $d_{\nabla,H}$ becomes
\begin{equation*}
d_{\nabla,H} \left( \begin{matrix} \omega_i \\ \omega_{i-1} \end{matrix} \right) = \left( \begin{matrix} d\omega_i + H_3 \wedge \omega_i + (-1)^{i-1} F \wedge \omega_{i-1} \\ d\omega_{i-1} + H_3 \wedge \omega_{i-1} + (-1)^i H_2 \wedge \omega_i \end{matrix} \right).
\end{equation*}


\subsection{T-duality in terms of differential forms}\label{tddf}

Let $\epsilon = w_1(L) \in H^1(M,\mathbb{Z}_2)$ be the first Stiefel-Whitney class of the flat orthogonal line bundle $L$. We now suppose that $X$ is equipped with an $\epsilon$-twisted gerbe $\mathcal{G}$ classified by a pair $(\pi^* \alpha , h) \in H^1(X,\mathbb{Z}_2) \times H^3(X,\mathbb{Z}_\epsilon)$, where $\alpha \in H^1(M,\mathbb{Z}_2)$. In other words $(\pi : X \to M , \epsilon , (\alpha , h))$ is a T-duality background. Let $H \in \Omega^3(X,L)$ be a $3$-form representing the image of $h$ in real cohomology. We can take $H$ to be an invariant representative. Thus $H$ has a decomposition of the form
\begin{equation}\label{defh}
H = \pi^*(H_3) + \pi^*(\hat{F}) \wedge A
\end{equation}
where $H_3 \in \Omega^3(M,L)$, $\hat{F} \in \Omega^2(M , L \otimes V^*)$. In Section \ref{invco} $\hat{F}$ was instead denoted by $H_2$. The change in notation reflects the fact that we are heading towards T-duality. The condition that $H$ is closed is equivalent to
\begin{eqnarray}
d_{L \otimes V^*} \hat{F} &=& 0 \label{fhat} \\
d_{L} H_3 + \hat{F} \wedge F &=& 0 \label{ffhat}
\end{eqnarray}
where for a flat line bundle $W$ we let $d_W$ denote the differential on $W$-valued forms. We see from (\ref{fhat}) that $\hat{F}$ determines a cohomology class $[\hat{F}] \in H^2(M , L \otimes V^*)$ and by (\ref{ffhat}) we have $[\hat{F}] \wedge [F] = 0 \in H^4(M,L)$.\\

Let $(\hat{\pi} : \hat{X} \to M , \epsilon , (\hat{\alpha} , \hat{h}))$ be a T-duality background which is T-dual to $(X,\pi,\epsilon,(\alpha , h))$. Let $\hat{\xi}$ be the Stiefel-Whitney class of $\hat{X} \to M$ and $\hat{c} = c_1(\hat{X})$ the twisted Chern class. Let $\mathbb{Z}_{\hat{\xi}}$ be the $\mathbb{Z}$-valued local system determined by $\hat{\xi}$ and $\hat{V} = \mathbb{Z}_{\hat{\xi}} \otimes \mathbb{R}$ the corresponding flat orthogonal bundle. 

\begin{proposition}\label{triple}
Let $(\hat{\pi} : \hat{X} \to M , \epsilon , (\hat{\alpha} , \hat{h}))$ be a T-duality background T-dual to $(X,\pi,\epsilon,(\alpha , h))$. There is a connection $\hat{A}$ on $\hat{X}$ with curvature $\hat{F}$ and such that $\hat{h}$ is represented by the invariant $3$-form $\hat{H} \in \Omega^3(\hat{X},L)$ defined as follows
\begin{equation}\label{defhhat}
\hat{H} = \hat{\pi}^*(H_3) + \hat{\pi}^* (F) \wedge \hat{A}.
\end{equation}
\begin{proof}
By definition of T-duality we have that $\hat{c} = \pi_*(h)$. However it follows from (\ref{defh}) that $\hat{F}$ is a representative for the image of $\pi_*(h)$ in real cohomology. Thus $\hat{F}$ represents $\hat{c}$ in real cohomology. It is straightforward to show that we can find a connection $\hat{A} \in \Omega^1(\hat{X},\hat{V})$ for $\hat{X}$ such that $\hat{F}$ is the curvature of $\hat{A}$, $d\hat{A} = \hat{\pi}^* \hat{F}$.

Define an invariant $3$-form $\hat{H} \in \Omega^3(\hat{X},L)$ by (\ref{defhhat}). Let $\hat{h}_\mathbb{R} \in H^3(\hat{X},L)$ be the image of $\hat{h}$ in real cohomology and define $d = \hat{h}_\mathbb{R} - [\hat{H}]$. We find that $\hat{\pi}_*(d) = 0$ and $\hat{p}^* d = 0$. Lemma \ref{mainlem} can easily be adapted to real coefficients (in fact the proof is considerably simpler over $\mathbb{R}$), so we find that there exists a class $e \in H^1(M,\hat{V})$ such that $d = \hat{\pi}^*( c \smallsmile e)$. Let $e' \in \Omega^1(M,\hat{V})$ be a representative for $e$. We find that $\hat{h}_\mathbb{R} = [\hat{H} + \hat{\pi}^*( F \wedge e')]$. Let us now replace the connection $\hat{A}$ by $\hat{A}' = \hat{A} + \hat{\pi}^* e'$. Since $e'$ is closed we still have that $\hat{F}$ is the curvature of $\hat{A}'$ but now $\hat{\pi}^*(H_3) + \hat{\pi}^*(F) \wedge \hat{A}'$ is a representative for $\hat{h}_\mathbb{R}$.
\end{proof}
\end{proposition}

\begin{definition}
Let $(X,\pi,\epsilon,(\alpha,h))$ be a T-duality background and $ \linebreak(\hat{H},\hat{\pi},\epsilon,(\hat{\alpha},\hat{h}))$ a T-dual. A {\em differential T-duality triple} is a triple $(A,\hat{A},H_3)$ where
\begin{itemize}
\item{$A$ is connection on $X$}
\item{$\hat{A}$ is a connection on $\hat{X}$}
\item{$H_3$ is an $L$-valued $3$-form on $M$}
\end{itemize}
such that if $F,\hat{F}$ are the curvatures of $A,\hat{A}$ then
\begin{itemize}
\item{$d_L H_3 + \hat{F} \wedge F = 0$}
\item{$H = \pi*(H_3) + \pi^*(\hat{F}) \wedge A$ represents $h$ in real cohomology}
\item{$\hat{H} = \hat{\pi}^*(H_3) + \hat{\pi}^*(F) \wedge \hat{A}$ represents $\hat{h}$ in real cohomology}
\end{itemize}
\end{definition}

Proposition \ref{triple} can be rephrased as saying that for any pair of T-dual backgrounds a differential T-duality triple $(A,\hat{A},H_3)$ exists. For a differential T-duality triple we define on the correspondence space $C = X \times_M \hat{X}$ an $L$-valued $2$-form given by
\begin{equation}\label{bdef}
\mathcal{B} = A \wedge \hat{A}.
\end{equation}
It follows that
\begin{equation}\label{hhhat}
\hat{p}^* (\hat{H}) - p^*(H) = d \mathcal{B}.
\end{equation}


\subsection{T-duality of twisted cohomology}\label{tdtc}

As in Section \ref{tddf} let $(X,\pi,\epsilon,(\alpha,h))$, $(\hat{X},\hat{\pi},\epsilon,(\hat{\alpha},\hat{h}))$ be T-dual backgrounds and $(A,\hat{A},H_3)$ a T-duality triple. To such data we will construct a T-duality transformation $T : \Gamma( S^{i,\alpha}_\epsilon(X)) \to \Gamma( S^{i-1,\hat{\alpha}}_\epsilon(\hat{X}))$ and prove that it induces an isomorphism of twisted cohomologies.

\begin{definition}
The {\em T-duality transformation} $T : \Gamma( S^{i,\alpha}_\epsilon(X)) \to \Gamma( S^{i-1,\hat{\alpha}}_\epsilon(\hat{X}))$ associated to a differential T-duality triple is given by
\begin{equation}\label{ttran}
T \omega = \int_{C / \hat{X}} e^{-\mathcal{B}} \wedge p^*( \omega )
\end{equation}
\end{definition}
where $\mathcal{B}$ is defined as in (\ref{bdef}) and the integral $\int_{C / \hat{X}}$ in (\ref{ttran}) is a fibre integration. If $f : Y \to Z$ is a circle bundle with Stiefel-Whitney class $w \in H^1(Z,\mathbb{Z}_2)$ and $V = \mathbb{Z}_w \otimes \mathbb{R}$ is the associated orthogonal flat line bundle then the fibre integration is a map $\int_{Y / Z} : \Omega^i(Y,A) \to \Omega^{i-1}(Z,A \otimes V)$, where $A$ any auxiliary flat line bundle. Let $A$ be a connection form for $Y \to Z$ which as usual can be viewed as an invariant $V$-valued $1$-form on $Y$. Then fibre integration has the following properties
\begin{eqnarray*}
\int_{Y / Z} f^*(\omega) \wedge A &=& \omega, \\
\int_{Y / Z} f^*(\omega') &=& 0.
\end{eqnarray*}

In fact these properties completely determine $\int_{Y / Z}$ on the subspace of invariant forms which suffices for our purposes. In addition we observe the following properties
\begin{eqnarray*}
\int_{Y / Z} f^*(a) \wedge b &=& f^*(a) \wedge \int_{Y / Z} b \\
d \int_{Y / Z} a &=& \int_{Y / Z} da.
\end{eqnarray*}
We see immediately from these properties and Equation (\ref{hhhat}) that
\begin{equation*}
d_{\nabla , \hat{H}} (T \omega) = T (d_{\nabla , H} \omega)
\end{equation*}
for any $\omega \in \Gamma( S^{i,\alpha}_\epsilon(X))$. Thus $T$ descends to a map $T_* : H^{i,\alpha,H}_\epsilon(X) \to H^{i-1,\hat{\alpha},\hat{H}}_\epsilon(\hat{X})$ in cohomology.

\begin{proposition}\label{tdciso}
The map $T_* : H^{i,\alpha,H}_\epsilon(X) \to H^{i-1,\hat{\alpha},\hat{H}}_\epsilon(\hat{X})$ is an isomorphism.
\begin{proof}
It is straightforward to see that $T$ sends invariant forms to invariant forms, so defines a similar map $T_* : H^{i,\alpha,H}_{\epsilon,{\rm inv}}(X) \to H^{i-1,\hat{\alpha},\hat{H}}_{\epsilon,{\rm inv}}(\hat{X})$ between invariant cohomologies. By Proposition \ref{tciso} it suffices to show that this map is an isomorphism. In fact we can prove there is a commutative diagram
\begin{equation*}\xymatrix{
H^{i,\alpha,H}_{\epsilon,{\rm inv}}(X) \ar[dr]^{\simeq} \ar[rr]^{T_*} & & H^{i-1,\hat{\alpha},\hat{H}}_{\epsilon,{\rm inv}}(\hat{X}) \ar[dl]_{\hat{T}_*} \\
& H^{i-2,\alpha+\epsilon,H}_{\epsilon,{\rm inv}}(X) & 
}
\end{equation*}
Here $\hat{T}$ is the T-duality transformation obtained by reversing the roles of $X$ and $\hat{X}$ while $H^{i,\alpha,H}_{\epsilon,{\rm inv}}(X) \simeq H^{i-2,\alpha+\epsilon,H}_{\epsilon,{\rm inv}}(X)$ is the canonical isomorphism that follows directly from the definition of $S^{i,\alpha}_\epsilon(X)$.

Recall the isomorphism (\ref{decom}) relating invariant forms to forms on the base. Under this isomorphism we find that $T$ is the map $T : \Gamma( S^{i,\alpha}_\epsilon(M) ) \oplus \Gamma( S^{i-1,\alpha \otimes V}_\epsilon(M)) \to \Gamma( S^{i-1,\hat{\alpha}}_\epsilon(M) ) \oplus \Gamma( S^{i-2,\hat{\alpha} \otimes V}_\epsilon(M))$ given by
\begin{equation*}
T \left( \begin{matrix} \omega_i \\ \omega_{i-1}  \end{matrix} \right) = \left( \begin{matrix} \omega_{i-1} \\ \omega_i \end{matrix} \right).
\end{equation*}
The above commutative diagram is now trivial to verify and proves that $T_*$ is an isomorphism.
\end{proof}
\end{proposition}


\subsection{T-duality of conformal Courant algebroids}\label{tdcca}

Associated to the T-duality background $(X,\pi,\epsilon,(\alpha,h))$ is an exact conformal Courant algebroid $E$ on $X$. Up to isomorphism $E$ is given by $E = TX \oplus (L \otimes T^*X)$ with the twisted Dorfman bracket $[ \, , \, ]_H$. We observe that since $H$ is invariant we can speak of invariant sections of $E$ and that the space of invariant sections $\Gamma_{\rm inv}(E)$ is Closed under $[ \, , \, ]_H$. It is not hard to see that the invariant sections thus define a conformal Courant algebroid $\mathcal{E}(X,H)$ on $M$.\\

Let $(X,\pi,\epsilon,(\alpha,h))$, $(\hat{X},\hat{\pi},\epsilon,(\hat{\alpha},\hat{h}))$ be T-dual backgrounds. We will show that there is an isomorphism $\mathcal{E}(X,H) \simeq \mathcal{E}(\hat{X},\hat{H})$. More precisely let $(A,\hat{A},H_3)$ be a differential T-duality triple for $X,\hat{X}$. The data of the triple will define a bundle isomorphism $\phi : \mathcal{E}(X,H) \to \mathcal{E}(\hat{X},\hat{H})$ and we show this is an isomorphism of conformal Courant algebroids. Indeed we can use the connection $A$ to split the tangent bundle of $X$, $TX = TM \oplus V$ and similarly use $\hat{A}$ to obtain a splitting $T\hat{X} = TM \oplus \hat{V}$. Under these splittings we have vector bundle isomorphisms 
\begin{eqnarray*}
\mathcal{E}(X,H) &=& TM \oplus V \oplus (L \otimes V) \oplus (L \otimes T^*M), \\
\mathcal{E}(\hat{X},\hat{H}) &=& TM \oplus \hat{V} \oplus (L \otimes \hat{V}) \oplus (L \otimes T^*M).
\end{eqnarray*}
Next recall that by the definition of T-duality we have an isomorphism $\hat{V} = V \otimes L$ as flat line bundles and similarly $L \otimes \hat{V} = V$. Using these identifications we immediately get a bundle isomorphism $\phi : \mathcal{E}(X,H) \to \mathcal{E}(\hat{X},\hat{H})$.

\begin{proposition}\label{ccaiso}
The map $\phi : \mathcal{E}(X,H) \to \mathcal{E}(\hat{X},\hat{H})$ is an isomorphism of conformal Courant algebroids.
\begin{proof}
Recall as in Section \ref{twco} that there is an action $\gamma : E \otimes S^{i,\alpha}_\epsilon(X) \to S^{i-1,\alpha}_\epsilon(X)$. Clearly this action descends to an action of $\mathcal{E}(X,H)$ on invariant forms. A straightforward computation shows that for any $\omega \in \Gamma_{\rm inv}( S^{i,\alpha}_\epsilon(X) )$ and any section $a \in \Gamma(\mathcal{E}(X,H))$ we have
\begin{equation}\label{tandgam}
T( \gamma_a \omega) = \gamma_{\phi(a)} T\omega.
\end{equation}
Next recall from Equation (\ref{db}) that the Dorfman bracket is a derived bracket:
\begin{equation*}
\gamma_{[a,b]_H} = [ [d_{\nabla,H} , \gamma_a] , \gamma_b ].
\end{equation*}
We take $a,b$ to be invariant, that is sections of $\mathcal{E}(X,H)$, apply $T$ to both sides of this equation and use (\ref{tandgam}) to arrive at
\begin{equation*}
\gamma_{\phi([a,b]_H)} = [ [d_{\nabla,\hat{H}} , \gamma_{\phi(a)}] , \gamma_{\phi(b)} ].
\end{equation*}
Using (\ref{db}) again we have that the right hand side of the above equation is equal to $\gamma_{[\phi(a),\phi(b)]_{\hat{H}}}$ so we have shown $\phi( [a,b]_H ) = [\phi(a),\phi(b)]_{\hat{H}}$. To complete the proof one must show that $\phi$ preserves the bilinear pairings and anchors. This is straightforward to check.
\end{proof}
\end{proposition}


\section{Twisted $KR$-theory}\label{twkr}

We have already remarked that for string theory defined on an orientifold the flux for the Ramond-Ramond fields is thought to lie in twisted $KR$-theory. Complementing this it is thought that D-branes and orientifold planes carry Ramond-Ramond charges which are also elements of twisted $KR$-theory. The connection between orientifolds and $KR$-theory was observed in \cite{wit}, \cite{bgs} and a proposed definition of twisted $KR$-theory on orientifolds with torsion $H$-flux is given in \cite{braste}. The case of non-torsion $H$-flux and the refinement to differential twisted $KR$-theory can be found in \cite{dfm0}. A proper definition of twisted $KR$-theory is not presented there but it seems reasonable to expect that a definition along the lines given in \cite{fht} for complex $K$-theory easily extends to the $KR$ setting.

Based on this link between orientifolds and $KR$ theory we expect that T-dual backgrounds have isomorphic twisted $KR$-theories. More precisely we expect that the T-duality isomorphism in twisted cohomology $T_* : H^{i,\alpha,H}_\epsilon(X) \linebreak \to H^{i-1,\hat{\alpha},\hat{H}}_\epsilon(\hat{X})$ defined in Section \ref{tdtc} lifts to an isomorphism of the form $T_* : KR^{i,\alpha,h}_\epsilon(X) \to KR^{i-1,\alpha,\hat{h}}_\epsilon(\hat{X})$ and that there should be a Chern character relating twisted $KR$-theory to our twisted cohomology. In the case $\epsilon = 0$ the twisted $K$-theory reduces to complex $K$-theory with twists by $H^1(X,\mathbb{Z}_2) \times H^3(X,\mathbb{Z})$ as defined in \cite{fht}. For this case a T-duality isomorphism in twisted $K$-theory has been proved for principal circle bundles \cite{bem}, \cite{bunksch}, principal torus bundles \cite{bhm1}, \cite{brs}, general circle bundles \cite{bar} and general affine torus bundles \cite{bar1}. 

As will be shown in \cite{bar1} the case of affine torus bundles has a subtlety in that if $(\pi : X \to M , (\alpha , h)),(\hat{\pi} : \hat{X} \to M , (\hat{\alpha} , \hat{h}))$ are T-dual rank $n$ affine torus bundles then there is an isomorphism in twisted $K$-theory of the form $K^{i,\alpha,h}(X) \simeq K^{i-n,\hat{\alpha}, \hat{h} + W_3(V) + \beta( w_1(V) \alpha)}(\hat{X})$ where $V$ is the vertical bundle of $\pi : X \to M$ and $\beta : H^2(M,\mathbb{Z}_2) \to H^3(M,\mathbb{Z})$ is the Bockstein homomorphism. If $V$ satisfies the topological constraint $W_3(V) + \beta( w_1(V) \alpha) = 0$ then we get an honest T-duality isomorphism $K^{i,\alpha,h}(X) \simeq K^{i-n,\hat{\alpha}, h}(\hat{X})$. Alternatively it is possible to modify the definition of T-duality slightly so that the expression $\hat{h} + W_3(V) + \beta( w_1(V) \alpha)$ gets replaced by simply $\hat{h}$. At any rate in the case where $\epsilon$ is non-zero we expect that there will be a similar isomorphism in twisted $KR$-theory. In fact a proof along the lines of \cite{bar} ought to easily generalize to twisted $KR$-theory as soon as one develops Mayer-Vietoris and a push-forward. Rather than attempt to develop the necessary tools here we will simply provide some evidence to back up our claims.\\

For $\epsilon \in H^1(X,\mathbb{Z}_2)$ let $X_\epsilon \to X$ be the double cover and $KR^i_\epsilon(X)$ the $KR$-theory with respect to the corresponding involution. We denote the involution by $\sigma : X_\epsilon \to X_\epsilon$. The twisted cohomology groups $H^{i,\alpha,H}_\epsilon(X)$ defined in Section \ref{twco} are $4$-periodic whereas the $KR$-groups are $8$-periodic in general. On the other hand we are considering the special case of a free involution so some simplification should occur. As an indication of this we have the following.
\begin{proposition}
Suppose $\beta(\epsilon) = 0 \in H^2(X,\mathbb{Z})$ where $\beta$ is the Bockstein homomorphism. Then $KR^{i+4}_\epsilon(X) = KR^i_\epsilon(X)$.
\begin{proof}
Let $L$ be the flat orthogonal line bundle on $X$ corresponding to $\epsilon$. The condition $\beta(\epsilon) = 0$ is equivalent to the statement that the complexification $L \otimes \mathbb{C}$ is topologically trivial. Under $\mathbb{Z}_2 \to {\rm U}(1)$ we have that $L \otimes \mathbb{C}$ has a natural Hermitian form preserved by the flat connection. Let $s$ be a unit length section of $L \otimes \mathbb{C}$ on $X$. On the other hand when $L$ is pulled back to $X_\epsilon$ we can find a unit length covariantly constant section $s'$. Writing $s' = f s$ we see that $f : X_\epsilon \to {\rm U}(1)$ has the property $\sigma^*(f) = -f$. The trivial bundle $X_\epsilon \times \mathbb{C}$ on $X_\epsilon$ admits a lift $\tau$ of $\sigma$ defined by $\tau( x , z) = (\sigma(x) , f(x) \overline{z} )$. We see that $\tau$ is anti-linear and $\tau^2 = -1$ so $(X_\epsilon \times \mathbb{C} , \tau )$ defines an element $\tau \in KH^0(X_\epsilon)$ of quaternionic $K$-theory \cite{du}. Recall from \cite{du} that $KH^i(X_\epsilon) = KR^{i+4}(X_\epsilon)$. In particular we can think of $\tau$ as an element of $KR^4(X_\epsilon)$. Multiplication by $\tau$ gives a map $\tau : KR^i(X_\epsilon) \to KR^{i+4}(X_\epsilon)$ which is an isomorphism, since $\tau$ is invertible.
\end{proof}
\end{proposition}

In the case of a double cover $X_\epsilon \to X$ there is a spectral sequence for $KR$-theory generalizing the Atiyah-Hirzebruch spectral sequence in complex $K$-theory. A proof is sketched in \cite{brs} and in \cite{cobr}. In fact such a spectral sequence generalizes to involutions that are not free \cite{dug}. The spectral sequence is also applicable to twisted $KR$-theory, but we will content ourselves with a proof of the untwisted case as an illustration that twisted $KR$-theory is related to our twisted cohomology.

\begin{proposition}\label{specseq}
There is a spectral sequence $E_r^{p,q}$ for $KR^n_\epsilon(X)$ with the property that $E_3^{p,q} = H^p( X , \mathbb{Z}_{(q/2) \epsilon})$ if $q$ is even and $E_3^{p,q} = 0$ if $q$ is odd. Here $\mathbb{Z}_{(q/2) \epsilon}$ is the $\mathbb{Z}$-valued local system corresponding to $(q/2)\epsilon \in H^1(X,\mathbb{Z}_2)$.
\begin{proof}
Let $X^k$ be the $k$-skeleton for a CW-structure on $X$ and set $X^{-1} = \emptyset$. If $X_\epsilon \to X$ is the double cover of $X$ associated to $\epsilon$ we let $X^k_\epsilon$ be the inverse image of $X^k$ in $X_\epsilon$. It is not hard to see that $X_\epsilon \to X$ induces CW-structure on $X_\epsilon$ such that $X_\epsilon^k$ is the $k$-skeleton. In the usual fashion of the Atiyah-Hirzebruch spectral sequence the filtration $X_\epsilon^0 \subseteq X_\epsilon^1 \subseteq \dots \subseteq X_\epsilon$ induces a spectral sequence $\{ E_r^{p,q} , d_r \}$ that abuts to $KR^n_\epsilon(X)$. We find that $E_1^{p,q} = KR^{p+q}(X_\epsilon^p , X_\epsilon^{p-1})$. Next observe that $X^p_\epsilon/X_\epsilon^{p-1}$ and $X^p/X^{p-1}$ are bouquets of spheres. In $X^p/X^{p-1}$ there is a sphere for every $p$-cell of $X$ while in $X^p_\epsilon/X_\epsilon^{p-1}$ there are exactly two spheres for each sphere of $X^p/X^{p-1}$. The involution on $X^p_\epsilon/X^{p-1}_\epsilon$ simply permutes these spheres accordingly. Thus $KR^{p+q}(X_\epsilon^p , X_\epsilon^{p-1})$ can be replaced by 
\begin{equation*}
KR^{p+q}( \amalg (D^p \amalg D^p) , \amalg (S^{p-1} \amalg S^{p-1}) )
\end{equation*}
where we take a copy of $D^p \amalg D^p$ for each $p$-cell of $X$. Here $D^p$ is the $p$-dimensional disc, $S^{p-1}$ the boundary and the involution on $D^p \amalg D^p$ swaps the two discs. By suspension and $(1,1)$-periodicity \cite{at} this can be reduced to $KR^0(\amalg (D^{2p+q} \amalg D^{2p+q}) , \amalg (S^{2p+q-1} \amalg S^{2p+q-1}) )$, except that this time the involution swaps discs and acts as $-1$ in $p+q$ coordinates. It is now easy to see that each $p$-cell contributes a summand of the form $\widetilde{K}^0(S^{2p+q})$ (reduced complex $K$-theory). If $q$ is odd this group is zero while for $q$ even it is $\mathbb{Z}$. 

To work out the next stage in the spectral sequence however we need to think of this group in terms of local systems. If one goes around a loop that exchanges fibres of $X_\epsilon \to X$ then the effect is like that on $\widetilde{K}^0(S^{2p+q})$ which acts on a complex vector bundle $E$ by taking the complex conjugate $\overline{E}$ and then pulling back under the map on $S^{2p+q}$ which acts as $-1$ in $p+q$ coordinates. When $q$ is even we find the effect of such a change is to act on $\widetilde{K}^0(S^{2p+q}) = \mathbb{Z}$ as multiplication by $(-1)^{p+q/2 + p+q} = (-1)^{q/2}$. To proceed to the $E_2$ stage we note that through the Chern character we can compare the spectral sequence $E_r^{p,q}$ with a similar spectral sequence in cohomology. This immediately determines the differential $d_1 : E_1^{p,q} \to E_1^{p+1,q}$. In fact for fixed $q$ we see that $\{ E_1^{p,q} , d_1 \}_{p \in \mathbb{Z}}$ forms the cellular complex on $X$ for the local system $\mathbb{Z}_{(q/2)\epsilon}$ when $q$ is even. Thus we have that $E_2^{p,q} = H^{p}(X,\mathbb{Z}_{(q/2)\epsilon})$ when $q$ even and $E_2^{p,q} = 0$ when $q$ is odd. It follows immediately that the differentials $d_2$ all vanish since they change the parity of $q$, thus the $E_3$ stage coincides with the $E_2$ stage completing the proof.
\end{proof}
\end{proposition}

Observe that on tensoring by $\mathbb{R}$ we obtain exactly the $E_3$ term in the spectral sequence for twisted cohomology obtained in Section (\ref{twco}), where $L$ is the line bundle corresponding to $\epsilon$ and we take $\alpha$ and $H$ to be trivial. Let $E \in KR^0_\epsilon(X)$ be a real vector bundle and choose a unitary connection $\nabla$ on $E$ which preserves the real structure. The Chern character defined by $\nabla$ determines a map $KR^0_\epsilon(X) \to H^{0,0,0}_\epsilon(X)$. We can use suspension isomorphisms to extend this to a map of the form $KR^i_\epsilon(X) \to H^{i,0,0}_\epsilon(X)$, which we can think of as the Chern character for $KR$-theory of a double cover. We expect that this can be extended to a twisted Chern character for twisted $KR$-theory with values in the twisted cohomology $H^{i,\alpha,H}_\epsilon(X)$.\\

\begin{example}\label{ex1} Let us consider the simplest possible non-trivial example of $T$-duality in $KR$-theory. We take as our base the circle $M = S^1$. Let $\epsilon$ be the non-trivial class in $H^1(S^1,\mathbb{Z}_2) = \mathbb{Z}_2$. According to Definition \ref{tdd} the torus $X = T^2 \to S^1$ and the Klein bottle $\hat{X} = K \to S^1$ are T-duals, both equipped with trivial graded gerbes. By Proposition \ref{tdciso} we have an isomorphism of twisted cohomology: $H^{i,0,0}_\epsilon(T^2) \simeq H^{i-1,0,0}_\epsilon(K)$. Our conjecture is that this lifts to $KR$-theory so we expect a corresponding isomorphism $KR^i_\epsilon(T^2) \simeq KR^{i-1}_\epsilon(K)$. Applying the spectral sequence of Proposition \ref{specseq} we can easily calculate the twisted $KR$-theories of $T^2$ and $K$. The results are as follows:
\begin{equation*}
\renewcommand{\arraystretch}{1.4}
\begin{tabular}{|l|l|l|}
\hline
$i$ & $KR^i_\epsilon(T^2)$ & $KR^i_\epsilon(K)$ \\
\hline
$0$ & $\mathbb{Z} \oplus \mathbb{Z}_2$ & $\mathbb{Z}^2$ \\
$1$ & $\mathbb{Z}^2$ & $\mathbb{Z}$ \\
$2$ & $\mathbb{Z}$ & $\mathbb{Z}_2$ \\
$3$ & $\mathbb{Z}_2$ & $\mathbb{Z} \oplus \mathbb{Z}_2$ \\
\hline
\end{tabular}
\end{equation*}
in complete agreement with out expectations.
\end{example}


\appendix

\section{Leray-Serre spectral sequence for affine torus bundles}\label{aftbss}

Let $V$ be the $n$-dimensional vector space $V = \mathbb{R}^n$, $\Lambda = \mathbb{Z}^n$ the standard rank $n$ lattice in $V$ and $T^n = V/\Lambda$ the corresponding rank $n$ torus. Let ${\rm Aff}(T^n) = {\rm GL}(n,\mathbb{Z}) \ltimes T^n$ be the group of affine transformations of $T^n$ generated by translations and linear transformations in ${\rm GL}(n,\mathbb{Z})$. In \cite{bar} we say that a $T^n$-bundle $\pi : X \to M$ is {\em affine} if it is a locally trivial fibre bundle with $T^n$ fibres and transition maps valued in ${\rm Aff}(T^n)$. We now recall the classification of such bundles. First the projection ${\rm Aff}(T^n) \to {\rm GL}(n,\mathbb{Z})$ implies that to every affine torus bundle $X \to M$ is an associated principal ${\rm GL}(n,\mathbb{Z})$-bundle. Such bundles are in bijection with conjugacy classes of representations $\rho : \pi_1(M) \to {\rm GL}(n,\mathbb{Z})$ which we call the {\em monodromy} of $X$. 

Using the ${\rm GL}(n,\mathbb{Z})$ action on $\Lambda = \mathbb{Z}^n$ we have that associated to any representation $\rho : \pi_1(M) \to {\rm GL}(n,\mathbb{Z})$ is a local system $\Lambda_\rho$ with $\mathbb{Z}^n$ coefficients. Secondly associated to $X$ there is a degree $2$ class called the {\em twisted Chern class} which has values in a local system with $\mathbb{Z}^n$ coefficients. If one identifies the principal ${\rm GL}(n,\mathbb{Z})$-bundle associated to $X$ with a monodromy representation $\rho$ then the twisted Chern class can be identified with a class $c \in H^2(M,\Lambda_\rho)$. However the pair $(\rho,c)$ is not uniquely defined by $X$. Rather any pair $(\rho',c')$ related to $(\rho,c)$ by the natural action of ${\rm GL}(n,\mathbb{Z})$ is similarly associated to $X$.

It is shown in \cite{bar} that isomorphism classes of rank $n$ affine torus bundles over $M$ are in bijection with equivalence classes of pairs $(\rho,c)$ where $\rho : \pi_1(M) \to {\rm GL}(n,\mathbb{Z})$ is a monodromy representation and $c \in H^2(M,\Lambda_\rho)$ a twisted Chern class. The equivalence relation is as above: $(\rho,c) \sim (\rho' , c')$ if they are related by the natural action of ${\rm GL}(n,\mathbb{Z})$.\\

Let $\pi : X \to M$ be a rank $n$ affine torus bundle over $M$ and $A$ a local system on $M$ with $\mathbb{Z}$-coefficients. In \cite{bar} we determined the $E_2$ stage $E_2^{p,q}(\pi,A)$ of the Leray-Serre spectral sequence for $X \to M$ and the differential $d_2$. The terms in the $E_2$-stage are
\begin{equation*}
E_2^{p,q}(\pi,A) = H^p(M , A \otimes \wedge^q \Lambda_\rho^*).
\end{equation*}
To define the differential $d_2 : E_2^{p,q}(\pi,A) \to E_2^{p+2,q-1}(\pi,A)$ first let us note that the contraction map $\Lambda \otimes \wedge^q (\Lambda)^* \to \wedge^{q-1} \Lambda^*$ extends to a map of local systems $\Lambda_\rho \otimes \wedge^q \Lambda_\rho^* \to \wedge^{q-1} \Lambda^*_\rho$ and thus to cohomology $H^r(M , \Lambda_\rho) \otimes H^s(M,A \otimes \wedge^q \Lambda_\rho^*) \to H^{r+s}(M,A \otimes \wedge^{q-1} \Lambda_\rho^*)$. Setting $r=2, s = p$ and substituting the twisted Chern class into the first factor we get a map $H^{p}(M,A \otimes \wedge^q \Lambda_\rho^*) \to H^{p-2}(M,A \otimes \wedge^{q-1} \Lambda_\rho^*)$. We showed that this map is precisely the differential $d_2$. To be more precise one should fix sign conventions in defining the contraction in cohomology as done in \cite{bar}. For the present paper the level of detail here suffices.

\bibliographystyle{amsplain}

\end{document}